\documentclass[12pt,reqno,a4]{amsart}

\usepackage[latin1]{inputenc}
\usepackage{amsmath}
\usepackage{amsthm}
\usepackage{amssymb}
\usepackage{amsfonts}
\usepackage{enumerate}
\usepackage{stmaryrd}
\usepackage{enumitem}
\usepackage{centernot}
\usepackage{todonotes}
\usepackage{mathrsfs}
\usepackage{listings}
\usepackage{multicol}
\usepackage[noend]{algorithmic} 

\usepackage{array}
\newcolumntype{L}[1]{>{\raggedright\let\newline\\\arraybackslash\hspace{0pt}}m{#1}}
\newcolumntype{C}[1]{>{\centering\let\newline\\\arraybackslash\hspace{0pt}}m{#1}}
\newcolumntype{R}[1]{>{\raggedleft\let\newline\\\arraybackslash\hspace{0pt}}m{#1}}

\usepackage{varwidth}

\usepackage{tikz}
\usetikzlibrary{arrows}
\usetikzlibrary{decorations.pathmorphing}
\usetikzlibrary{decorations.pathreplacing}
\usetikzlibrary{matrix}
\usetikzlibrary{calc}
\usetikzlibrary{shapes}
\usetikzlibrary{patterns}
\usetikzlibrary{fit,backgrounds,scopes}

\makeatletter
\pgfdeclarepatternformonly[\LineSpace,\LineWidth]{my horizontal lines}%
{\pgfpointorigin}{\pgfqpoint{100pt}{1pt}}{\pgfqpoint{100pt}{\LineSpace}}%
{
\pgfsetcolor{\tikz@pattern@color}
\pgfsetlinewidth{\LineWidth}
\pgfpathmoveto{\pgfqpoint{0pt}{0.5pt}}
\pgfpathlineto{\pgfqpoint{100pt}{0.5pt}}
\pgfusepath{stroke}
}

\pgfdeclarepatternformonly[\LineSpace,\LineWidth]{my vertical lines}%
{\pgfpointorigin}{\pgfqpoint{1pt}{100pt}}{\pgfqpoint{\LineSpace}{100pt}}%
{
\pgfsetcolor{\tikz@pattern@color}
\pgfsetlinewidth{\LineWidth}
\pgfpathmoveto{\pgfqpoint{0.5pt}{0pt}}
\pgfpathlineto{\pgfqpoint{0.5pt}{100pt}}
\pgfusepath{stroke}
}

\pgfdeclarepatternformonly[\LineSpace,\LineWidth]{my grid}%
{\pgfqpoint{-1pt}{-1pt}}{\pgfqpoint{\LineSpace}{\LineSpace}}
{\pgfqpoint{\LineSpace}{\LineSpace}}%
{
\pgfsetcolor{\tikz@pattern@color}
\pgfsetlinewidth{\LineWidth}
\pgfpathmoveto{\pgfqpoint{0pt}{0pt}}
\pgfpathlineto{\pgfqpoint{0pt}{\LineSpace + 0.1pt}}
\pgfpathmoveto{\pgfqpoint{0pt}{0pt}}
\pgfpathlineto{\pgfqpoint{\LineSpace + 0.1pt}{0pt}}
\pgfusepath{stroke}
}

\pgfdeclarepatternformonly[\LineSpace,\LineWidth]{my north east lines}
{\pgfqpoint{-\LineWidth}{-\LineWidth}}{\pgfqpoint{\LineSpace}{\LineSpace}}
{\pgfqpoint{\LineSpace}{\LineSpace}}%
{
\pgfsetcolor{\tikz@pattern@color}
\pgfsetlinewidth{\LineWidth}
\pgfpathmoveto{\pgfqpoint{-\LineWidth}{-\LineWidth}}
\pgfpathlineto{\pgfqpoint{\LineSpace + 0.1pt}{\LineSpace + 0.1pt}}
\pgfusepath{stroke}
}

\pgfdeclarepatternformonly[\LineSpace,\LineWidth]{my north west lines}
{\pgfqpoint{-\LineWidth}{-\LineWidth}}{\pgfqpoint{\LineSpace}{\LineSpace}}
{\pgfqpoint{\LineSpace}{\LineSpace}}%
{
\pgfsetcolor{\tikz@pattern@color}
\pgfsetlinewidth{\LineWidth}
\pgfpathmoveto{\pgfqpoint{-\LineWidth}{\LineSpace}}
\pgfpathlineto{\pgfqpoint{\LineSpace + 0.1pt}{-\LineWidth}}
\pgfusepath{stroke}
}

\pgfdeclarepatternformonly[\LineSpace,\LineWidth]{my crosshatch}%
{\pgfqpoint{-1pt}{-1pt}}{\pgfqpoint{\LineSpace}{\LineSpace}}
{\pgfqpoint{\LineSpace}{\LineSpace}}%
{
\pgfsetcolor{\tikz@pattern@color}
\pgfsetlinewidth{\LineWidth}
\pgfpathmoveto{\pgfqpoint{\LineSpace + 0.1pt}{0pt}}
\pgfpathlineto{\pgfqpoint{0pt}{\LineSpace + 0.1pt}}
\pgfpathmoveto{\pgfqpoint{0pt}{0pt}}
\pgfpathlineto{\pgfqpoint{\LineSpace + 0.1pt}{\LineSpace + 0.1pt}}
\pgfusepath{stroke}
}

\pgfdeclarepatternformonly[\LineSpace,\PointSize]{my dots}%
{\pgfqpoint{-\LineSpace*0.25}{-\LineSpace*0.25}}
{\pgfqpoint{\LineSpace*0.25}{\LineSpace*0.25}}
{\pgfqpoint{\LineSpace*0.75}{\LineSpace*0.75}}%
{
\pgfsetcolor{\tikz@pattern@color}
\pgfpathcircle{\pgfqpoint{0pt}{0pt}}{\PointSize}
\pgfusepath{fill}
}
\makeatother

\newdimen\LineSpace
\newdimen\PointSize
\newdimen\LineWidth
\tikzset{
line space/.code={\LineSpace=#1},
line space=3pt
}
\tikzset{
point size/.code={\PointSize=#1},
point size=.5pt
}
\tikzset{
pattern line width/.code={\LineWidth=#1},
pattern line width=.4pt
}

\newtheoremstyle{theoremstyle}
  {10pt}      
  {5pt}       
  {\itshape}  
  {}          
  {\bfseries} 
  {}         
  {\newline}      
  {}          
\newtheoremstyle{examplestyle}
  {10pt}      
  {5pt}       
  {}          
  {}          
  {\bfseries} 
  {}         
  {\newline}      
  {}          
\theoremstyle{theoremstyle}
\newtheorem{theorem}{Theorem}[section]
\newtheorem{lemma}[theorem]{Lemma}
\newtheorem{proposition}[theorem]{Proposition}
\newtheorem{corollary}[theorem]{Corollary}
 \theoremstyle{examplestyle}
 \newtheorem{example}[theorem]{Example}
 \newtheorem{definition}[theorem]{Definition}
 \newtheorem{remark}[theorem]{Remark}
 \newtheorem{algorithm}[theorem]{Algorithm}

\newcommand{\NN }{\mathbb{N}}

\newcommand{\RR }{\mathbb{R}}
\newcommand{\QQ }{\mathbb{Q}}
\newcommand{\ZZ }{\mathbb{Z}}
\newcommand{\ux}{{\mathbf{x}}}
\newcommand{\ov}{\overline}
\newcommand{\Rt}{{R\llbracket t\rrbracket}}
\newcommand{\Rtx}{{R\llbracket t\rrbracket [\ux]}}
\newcommand{\Rtxp}{{R[t,\ux]}}
\newcommand{\suchthat}{\;\ifnum\currentgrouptype=16 \middle\fi|\;}
\newcommand{\bigmid}{\left.\vphantom{\Big\{} \suchthat \vphantom{\Big\}}\right.}

\DeclareMathOperator{\syz}{syz}
\DeclareMathOperator{\Mon}{Mon}
\DeclareMathOperator{\lm}{LM}
\DeclareMathOperator{\lc}{LC}
\DeclareMathOperator{\lt}{LT}
\DeclareMathOperator{\tail}{tail}
\DeclareMathOperator{\HDDwR}{HDDwR}
\DeclareMathOperator{\initial}{in}
\DeclareMathOperator{\Witness}{Witness}
\DeclareMathOperator{\Lift}{Lift}
\DeclareMathOperator{\Flip}{Flip}
\DeclareMathOperator{\Mat}{Mat}
\DeclareMathOperator{\rec}{rec}

\newcommand{\lang}[1]{#1}
\newcommand{\kurz}[1]{}

\begin{document}

   \parindent0cm
   \parskip1ex

   \title[Gröbner Fans]{Gröbner Fans of $x$-homogeneous Ideals in $\Rtx$}
   \author{Thomas Markwig}
   \address{Technische Universit\"at Kaiserslautern\\
     Fachbereich Mathematik\\
     Erwin--Schr\"odinger--Stra\ss e\\
     D --- 67663 Kaiserslautern
     }
   \email{keilen@mathematik.uni-kl.de}
   \urladdr{http://www.mathematik.uni-kl.de/\textasciitilde keilen}
   \author{Yue Ren}
   \address{Technische Universit\"at Kaiserslautern\\
     Fachbereich Mathematik\\
     Erwin--Schr\"odinger--Stra\ss e\\
     D --- 67663 Kaiserslautern
     }
   \email{ren@mathematik.uni-kl.de}
   \urladdr{http://www.mathematik.uni-kl.de/\textasciitilde ren}
   \thanks{The first and the second author were supported by the German Israeli Foundation
     grant no 1174-197.6/2011 and by the German Research Foundation
     (Deutsche Forschungsgemeinschaft (DFG)) trough the Priority
     Programme 1489.}


   \date{December, 2015}

     
   \begin{abstract}
     We generalise the notion of Gr\"obner fan to ideals
     in $R\llbracket t\rrbracket[x_1,\ldots,x_n]$ for certain classes of
     coefficient rings $R$ and give a constructive
     proof that the Gr\"obner fan is a rational polyhedral fan. For this we
     introduce the notion of initially reduced standard bases and show
     how these can be computed in finite time. We deduce algorithms for
     computing the Gr\"obner fan, implemented in the computer
     algebra system \textsc{Singular}. The problem is motivated by the
     wish to compute tropical varieties over the $p$-adic numbers, which
     are the intersection of a subfan of a Gr\"obner fan as studied in this
     paper by some affine hyperplane, as shown in a forthcoming paper.
   \end{abstract}

   \maketitle

   \section{Introduction}

   Gr\"obner fans of ideals $I$ in the polynomial ring over a field
   were first introduced and studied by Mora and Robbiano
   in \cite{MR88} as an invariant associated to the ideal. The Gr\"obner
   fan of $I$ is a convex rational polyhedral fan classifying all
   possible leading ideals of $I$ w.r.t.~arbitrary global monomial
   orderings and encoding
   the impact of all these orderings on the ideal. It provides an
   interesting link between commutative algebra and convex geometry,
   opening the rich tool box of the latter for the first. Moreover,
   tropical varieties, which have gained lots of interest recently,
   can be described often as subcomplexes of certain Gr\"obner fans and
   can be computed that way. The latter is the main motivation for our
   paper, as we explain further down.

   Mora and Robbiano
   describe in their paper an algorithm to compute the Gr\"obner
   fan. The underlying structure was then used efficiently by Collart,
   Kalkbrenner and Mall in \cite{CKM97} to transform a standard basis
   w.r.t.~one global monomial ordering into a standard basis
   w.r.t.~another
   one by passing through several cones of the Gr\"obner fan. At a
   common facet of two cones a local change of the standard basis was necessary making
   use of the fact that the monomial orderings of the neighbouring
   cones can be seen as a refinement of a common partial ordering on
   the monomials.  Their
   methods were later refined by many others (see
   e.g.~\cite{AGK97,Tra00,Aue05,FJLT07,SS08}).

   For homogeneous ideals the Gr\"obner fan is complete
   and Sturmfels showed in \cite{Stu96} that it is the normal
   fan of a polytope, the state polytope of $I$. If the ideal is not
   homogeneous the Gr\"obner fan is in general neither complete, nor is the part in
   the positive orthant the normal fan of a polyhedron, as was shown
   by Jensen in \cite{Jen07a}.

   Since the notion of the Gr\"obner fan turned out to be so powerful
   in the polynomial ring it was in the sequel generalised to further
   classes of rings.
   Assi, Castro-Jim\'enez and Granger (see \cite{ACJG00}) and Saito,
   Sturmfels and Taka\-yama (see \cite{SST00}) studied an analogue of
   the Gr\"obner fan for ideals in the ring of algebraic differential
   operators. In a subsequent paper the first three authors
   generalised the notion to the ring of analytic differential
   operators (see \cite{ACJG01}), proving that the equivalence classes
   of weight vectors  yet again are convex rational polyhedral
   cones. Bahloul and Takayama (see \cite{BT06,BT07}) then show that
   these cones glue to give a fan and they give an algorithm to
   compute this fan. They show that their techniques apply to ideals
   in the subrings of convergent or formal power series over a field
   and treat this case explicitly. This leads to the notion of the local
   standard fan which covers the negative orthant and whose cones
   characterise the impact of the  local monomial orderings on the
   ideal in the power series ring.

   Even though the approach is
   algorithmical, it cannot be applied in practice right away,
   since the computation of the standard cones heavily relies  on the
   computation of a reduced standard basis, which even for polynomial
   input data in general contains power series and is not feasible in
   practice.  If the input data is polynomial
   Bahloul and Takayama, therefore, propose to homogenise the ideal, compute the
   Gr\"obner fan with the usual techniques and then to cut down the additional
   variable again. This will lead to a refinement of the actual local
   standard fan, but for each pair of neighbouring cones one can check
   with a standard basis computation, if the cones should be glued in
   the local standard fan. Since the number of fulldimensional cones in the refined
   fan may be larger by an order of magnitude, this approach is very
   expensive.

   In our paper we address a situation which in some respects is more
   general and in some is much more specialised than the above. It is
   motivated by a very particular application that we have in mind,
   the computation of tropical varieties over the $p$-adic numbers. These
   appear as the intersection of a subfan of the Gr\"obner fans
   studied in this paper with an affine hyperplane
   (see~\cite{MR15b}). Here we lay the theoretical and the
   algorithmical foundation for this approach to compute tropical
   varieties over the $p$-adic numbers, leading to the only currently
   available software for computing these varieties.

   In this paper we allow as coefficient domain a ring $R$ satisfying some
   additional technical properties which ensure that standard bases
   over $R$ can be computed (see Page~\ref{page:basering} and
   \cite{MRW15}). We then consider $\ux$-homogeneous ideals $I$ in the mixed power series
   polynomial ring $\Rtx=\Rt[x_1,\ldots,x_n]$,
   that is, we consider one local variable and any finite number of global
   variables. We then define the Gr\"obner fan of $I$ as usual with some
   necessary adjustments. The main theoretical result of this paper shows that the Gr\"obner
   fan is indeed a rational polyhedral fan covering all of the half space
   $\RR_{\leq 0}\times\RR^n$ (see Theorem~\ref{thm:groebnerFan}).
   For the theory the generators of $I$ may
   be arbitrary power series in the local variable, for the practice
   we restrict to input data which is polynomial in $t$ as well as in
   $\ux$, but homogeneity is only required w.r.t.~$\ux$.
   A major point when it comes to actually computing the Gr\"obner fans
   is that restricting to one local variable allows us to replace reduced
   standard bases by the weaker notion of initially reduced
   standard bases. We show  that these are sufficiently strong to
   let us read off the Gr\"obner cones (see
   Section~\ref{sec:groebnerfan}), yet weak enough to be computable
   for polynomial input data with a finite number of steps at the same
   time in important cases.

   Note that for polynomial input data we could have followed the
   approach of Bahloul and Takayama (see \cite{BT06,BT07}) by
   homogenising first, cutting down and gluing cones. However, not
   only is the gluing very costly, the Gr\"obner fan of the homogenised ideal
   has way more cones and these have plenty more facets that have to
   be traversed. For a simple tropical linear space in an example we
   have $20$ full-dimensional cones without homogenisation and $1393$
   for the homogenised ideal, and the number of facets that have to be
   traversed has increase by a factor way larger than $100$. Thus, already
   computing the Gr\"obner fan of the homogenised ideal is much more
   expensive than computing the Gr\"obner fan of $I$ directly via our
   approach.

   In Section~\ref{sec:basic} we introduce the basic notions used
   throughout the paper and we show that also in our situation there
   are only finitely many possible leading
   ideals. Section~\ref{sec:groebnerfan} is devoted to proving that
   the Gr\"obner fan is a rational polyhedral fan. We provide a
   constructive approach for the Gr\"obner cones using
   initially reduced standard bases.
   In Section~\ref{sec:initiallyreduced} we present algorithms to reduce standard bases
   initially in finite time under some additional hypotheses on $R$
   and the ideal (see Page~\pageref{page:initialReduction}), and in
   Section~\ref{sec:computation} we finally provide
   algorithms to compute Gr\"obner fans
   of ${\ux}$-homogeneous ideals, where for the latter we follow the
   lines of \cite{FJT07}. The algorithms are implemented in
   and distributed with \textsc{Singular} and they complement the
   software package \texttt{gfan}
   (see \cite{gfan}) by Jensen which is specialised in computing Gr\"obner
   fans for ideals in polynomial rings and their tropical varieties.

   \section{Basic notions}\label{sec:basic}

   Throughout this paper we assume that $R$ is a noetherian ring and
   that linear equations in $R$ are solvable,\label{page:basering}
   that is, for any choice of
   $c_1,\ldots,c_k\in R$ we can decide
   the ideal membership problem $b\in\langle c_1,\ldots,c_k\rangle$,
   if applicable represent $b$ as $b = a_1\cdot c_1+\cdots+a_k\cdot c_k$, and compute
   a finite generating set of the syzygy module
   $\syz_R(c_1,\ldots,c_k)$. The most important example that we have
   in mind is the ring of integers. For further classes of interesting  examples see
   \cite[Ex.~1.2]{MRW15}. Due to \cite{MRW15} this assumption ensures that in
   the mixed power series polynomial ring
   \begin{displaymath}
     \Rtx := R\llbracket t \rrbracket [x_1,\ldots,x_n],
   \end{displaymath}
   with a single variable~$t$, standard bases exist and are computable in finite
   time and with polynomial output, if the ideal is generated by
   polynomials.

   We represent an element $f$ of $\Rtx$ in the usual multiindex
   notation as
   \begin{displaymath}
     f=\sum_{\beta,\alpha}c_{\alpha,\beta}\cdot t^\beta\ux^\alpha
   \end{displaymath}
   with $\beta\in\NN$ and $\alpha=(\alpha_1,\ldots,\alpha_n)\in\NN^n$
   where $\ux^\alpha=x_1^{\alpha_1}\cdots x_n^{\alpha_n}$, and we
   sometimes represent it as
   \begin{displaymath}
     f=\sum_{\alpha}g_\alpha\cdot\ux^\alpha
   \end{displaymath}
   with
   \begin{displaymath}
     g_\alpha=\sum_\beta c_{\alpha,\beta}\cdot t^\beta\in
     \Rt
   \end{displaymath}
   as an element in the polynomial ring in $\ux$ over the ring
   $\Rt$. We then call $f$  \emph{$x$-homogeneous} if all monomials
   $\ux^\alpha$ have the same degree, and we call an ideal
   $I\unlhd\Rtx$ \emph{$x$-homogeneous} if it is generated by
   $\ux$-homogeneous elements. In what follows we will construct
   Gr\"obner fans of ${\ux}$-homogeneous ideals $I\unlhd\Rtx$ as fans
   on the closed half space $\RR_{\leq 0}\times\RR^n$.

   Let us now fix some standard notation used in the context of
   standard bases and Gr\"obner fans.
   We denote by
   \begin{displaymath}
     \Mon(t,\ux)=\big\{t^\beta\cdot\ux^\alpha\;\bigm|\;\beta\in\NN, \alpha\in\NN^n\big\}
   \end{displaymath}
   the multiplicative semigroup of monomials in the variables $t$ and $\ux$. A \emph{monomial
     ordering} on $\Mon(t,\ux)$ is a total ordering $>$ which is
   compatible with the semigroup structure on $\Mon(t,\ux)$, and we
   call it \emph{$t$-local} if  $1>t$. The least monomial
   $t^\beta\ux^\alpha$ w.r.t.~a
   $t$-local monomial ordering $>$ occuring in $0\not=f\in\Rtx$ is called the
   \emph{leading monomial} $\lm_>(f)=t^\beta\ux^\alpha$ of $f$, the corresponding coefficient
   is its \emph{leading coefficient} $\lc_>(f)=c_{\beta,\alpha}$,
   the term $\lt_>(f)=\lc_>(f)\cdot\lm_>(f)$ is its \emph{leading
     term} and $\tail_>(f)=f-\lt_>(f)$ its \emph{tail}, and we set
   $\lt_>(0)=0$. We call the ideal
   \begin{displaymath}
     \lt_>(I):=\langle \lt_>(f)\;|\; f\in I\rangle\unlhd\Rtxp
   \end{displaymath}
   the \emph{leading ideal} of $I$ w.r.t.~$>$. Note, that it is an
   ideal generated by terms, but in general not by monomials, since
   $R$ is only a ring. However, as in the case of base fields the
   number of possible leading ideals w.r.t.~$t$-local monomial
   orderings is finite, which will essentially imply that the Gr\"obner
   fan of $I$ has only finitely many cones.
   The proof of is an adaptation of the proof of \cite[Thm.~4.1]{CLO05}.

   \begin{proposition}\label{prop:finiteLeadingIdeals}
     Any ${\ux}$-homogeneous ideal $I\unlhd\Rtx$ has only finitely many leading ideals.
   \end{proposition}
   \begin{proof}
     Observe that an element $g\in\Rtx$ has only finitely many possible
     leading terms, since there are only finitely many distinct monomials
     in ${\ux}$ and a leading term w.r.t.~a $t$-local monomial
     ordering has to have minimal power in $t$.

     Now assume there are infinitely many leading ideals. For each leading ideal $J$,
     let $>_J$ be a $t$-local monomial ordering such that $\lt_{>_J}(I)=J$.
     Set $\Delta_0 := \{>_J \mid J \text{ leading ideal of } I\}$, so that different orderings in $\Delta_0$
     yield different leading ideals. By our assumption, $\Delta_0$ is infinite.

     Let $G_1 \subseteq I$ be a finite ${\ux}$-homogeneous generating set of
     $I$ and set $\Sigma_1$ to be the union of all potential leading
     terms of elements of $G_1$.
     Then $\Sigma_1$ is finite and hence, by the pigeonhole principle,
     there must be infinitely many monomial orderings
     $\Delta_1\subseteq\Delta_0$ which agree on $\Sigma_1$.
     \cite[Cor.~2.8]{MRW15} now implies that
     if $G_1\subseteq I$ was a standard basis for one of them,
     it would be a standard basis for all of them.
     As this cannot be the case, given an ordering $>_1\;\in\Delta_1$
     there must be an element $g_2 \in I$ such that
     $\lt_{>_1}(g_2)\notin J_1:=\langle \lt_{>_1}(g) \mid g \in
     G_1\rangle$ with $J_1$ being independent from the ordering chosen.

     Since $I$ is ${\ux}$-homogeneous, we may choose $g_2$ to be
     ${\ux}$-homogeneous. Moreover, by computing a determinate division with
     remainder w.r.t.~$G_1$ and $>_1$, we may assume that no
     term of $g_2$ lies in $J_1$ (see e.g.~condition (DD2) in
     \cite[Alg.~1.13]{MRW15}). In particular,
     \begin{displaymath}
       \lt_{>}(g_2)\notin J_1:=\langle \lt_{>}(g) \mid g \in G_1\rangle \text{ for any ordering } >\;\in\Delta_1.
     \end{displaymath}
     Setting $G_2:=G_1\cup \{g_2\}$, we can repeat the entire process,
     and find an infinite subset of monomial orderings
     $\Delta_2\subseteq\Delta_1$ such that $G_2$ is either a standard
     basis for all of them or for none of them. Consequently, there is a
     $g_3 \in I$ such that $\lt_{>}(g_3)\notin J_2:=\langle \lt_{>}(g)
     \mid g \in G_2\rangle$ for all monomial orderings $>\;\in
     \Delta_2$.
     We thus obtain an infinite chain of strictly ascending ideals
     $J_1 \subsetneq J_2 \subsetneq \ldots$,
     which contradicts the ascending chain condition of our noetherian ring $\Rtxp$.
   \end{proof}

   A weight vector $w=(w_0,\ldots,w_n)\in\RR_{<0}\times\RR^n$ induces a
   partial ordering on $\Mon(t,\ux)$ via
   \begin{displaymath}
     t^\beta\ux^\alpha \geq t^\delta\ux^\gamma
     \quad:\Longleftrightarrow\quad
     w\cdot(\beta,\alpha)\geq w\cdot(\delta,\gamma),
   \end{displaymath}
   where ``$\cdot$'' denotes the canonical scalar product. Any
   monomial ordering $>$ on $\Mon(t,\ux)$ can be used as a tie breaker
   to refine this partial ordering to a $t$-local monomial ordering
   $>_w$. Given $w\in\RR_{<0}\times\RR^n$ we denote by
   \begin{displaymath}
     \initial_w(f)=\sum_{w\cdot (\beta,\alpha)\text{
         maximal}}c_{\beta,\alpha}\cdot t^\beta\ux^\alpha\in \Rtxp
   \end{displaymath}
   the \emph{initial form} of $f$ w.r.t.~$w$ and by
   \begin{displaymath}
     \initial_w(I)=\langle \initial_w(f)\;|\;f\in I\rangle\unlhd \Rtxp
   \end{displaymath}
   the \emph{initial ideal} of $I$.

   Initial ideals of $I$ can be used to define an equivalence relation
   on the space of weight vectors $\RR_{<0}\times\RR^n$, by setting
   \begin{displaymath}
     w \sim v \quad :\Longleftrightarrow \quad \initial_w(I) = \initial_v(I).
   \end{displaymath}
   We denote the closure the equivalence class of a weight vector
   $w\in\RR_{< 0}\times\RR^n$ in the Euclidean topology by
   \begin{displaymath}
     C_w(I) := \ov{ \{ v\in \RR_{< 0}\times\RR^n \mid \initial_v(I)=\initial_w(I) \}}
     \subseteq \RR_{\leq 0}\times\RR^n,
   \end{displaymath}
   and call it an \emph{interior Gr\"obner cone} of $I$. We then call the intersection of $C_w(I)$ with the boundary,
   \begin{displaymath}
     C_w^0(I) := C_w(I) \cap (\{0\}\times\RR^n),
   \end{displaymath}
   a \emph{boundary Gr\"obner cone} of $I$, and given any
   $t$-local monomial ordering $>$, we set
   \begin{displaymath}
     C_>(I) := \ov{ \{ v\in \RR_{<0}\times\RR^n \mid \initial_v(I)=\lt_>(I)\} }\subseteq \RR_{\leq 0}\times\RR^n.
   \end{displaymath}
   Finally, we refer to the collection
   \begin{displaymath}
     \Sigma(I):=\{C_w(I)\mid w\in \RR_{<0}\times\RR^n\} \cup \{ C_w^0(I)\mid w\in \RR_{<0}\times\RR^n\},
   \end{displaymath}
   of all cones as the \emph{Gr\"obner fan} of $I$. It is this object
   whose properties we want to study and that we want to compute.

   \begin{example}\label{ex:intersectionHomogeneitySpace}
     Consider the principal ideal
     \begin{math}
       I=\langle g\rangle \unlhd \ZZ\llbracket t\rrbracket [x,y]
     \end{math}
     with $g=tx^2+xy+ty^2$.
     Because $\initial_w(I)=\langle\initial_w(g)\rangle$ for any
     $w\in\RR_{< 0}\times\RR^2$ and $g$ is $(x,y)$-homogeneous, it is easy
     to see that every Gr\"obner cone of $I$ is invariant under
     translation by $(0,1,1)$. Its Gr\"obner fan divides the weight space
     $\RR_{\leq 0}\times\RR^2$ into three distinct maximal Gr\"obner
     cones, see Figure~\ref{fig:groebnerFanPrincipalIdeal}. Note that
     the two
     red maximal cones intersect each other solely in the boundary
     $\{0\}\times\RR^2$, while the third maximal cone intersects the
     boundary in codimension $2$.
     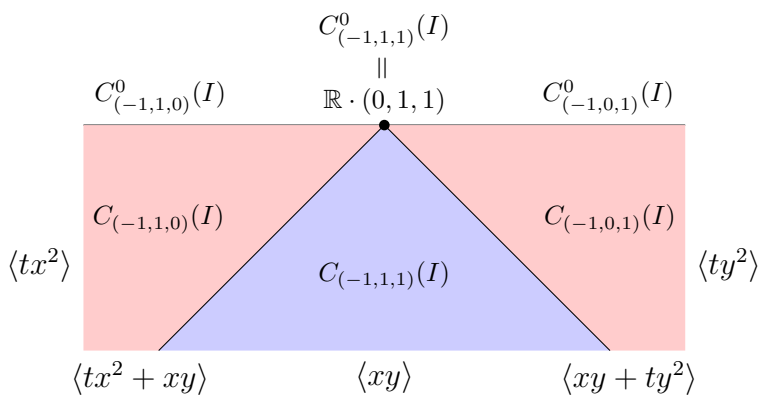
\begin{figure}[h]
       \centering
       \begin{tikzpicture}
         \draw (-4,0) -- (4,0);
         \fill [red!20] (-4,0) -- (0,0) -- (-3,-3) -- (-4,-3) -- cycle;
         \fill [red!20] (4,0) -- (0,0) -- (3,-3) -- (4,-3) -- cycle;
         \fill [blue!20] (0,0) -- (3,-3) -- (-3,-3) -- cycle;
         \fill (0,0) circle (2pt);
         \node [anchor=south, font=\footnotesize] (linspace) at (0,0) {$\RR\cdot (0,1,1)$};
         \node [anchor=south, font=\footnotesize, yshift=0.25cm] (C0) at (linspace.north) {$C_{(-1,1,1)}^0(I)$};
         \draw [draw opacity = 0] (linspace) -- node[sloped] {$=$} (C0);
         \draw (0,0) -- (-3,-3);
         \draw (0,0) -- (3,-3);
         \node [font=\footnotesize] at (-3,-1.25) {$C_{(-1,1,0)}(I)$};
         \node [anchor=north] at (0,-3) {$\langle xy \rangle$};
         \node [anchor=north east] at (-4,-1.5) {$\langle tx^2 \rangle$};
         \node [anchor=north west] at (4,-1.5) {$\langle ty^2 \rangle$};
         \node [font=\footnotesize] at (3,-1.25) {$C_{(-1,0,1)}(I)$};
         \node [anchor=north,xshift=-0.25cm] at (-3,-3) {$\langle tx^2+xy \rangle$};
         \node [anchor=north,xshift=0.25cm] at (3,-3) {$\langle xy+ty^2 \rangle$};
         \node [font=\footnotesize] at (0,-2) {$C_{(-1,1,1)}(I)$};
         \node [anchor=south west, font=\footnotesize] at (-4,0) {$C_{(-1,1,0)}^0(I)$};
         \node [anchor=south east, font=\footnotesize] at (4,0) {$C_{(-1,0,1)}^0(I)$};
       \end{tikzpicture}
       \caption{$\Sigma(\langle tx^2+xy+ty^2\rangle)$ projected along $\RR\cdot (0,1,1)$}
       \label{fig:groebnerFanPrincipalIdeal}
     \end{figure}
   \end{example}

   \lang{
   We will finish this section with two simple technical results on ideals in
   $\Rtxp$, well known for the case of base fields, which will be used
   when dealing with initial ideals.

   \begin{lemma}\label{lem:idealGeneratedByTerms0}
     Let $J\unlhd R[t,{\ux}]$ be an ideal generated by terms and let $f\in J$. Then each term of $f$ is again contained in $J$.
   \end{lemma}
   \begin{proof}
     Let $>$ be a $t$-local monomial ordering on $\Mon(t,{\ux})$, and let
     $p_1,\ldots,p_k\in R[t,{\ux}]$ be terms generating $J$.
     Hence there exist $q_1,\ldots,q_k
     \in R[t,{\ux}]$ such that
     \begin{displaymath}
       f=q_1\cdot p_1+\ldots+q_k\cdot p_k,
     \end{displaymath}
     and we may assume that $\lm_>(f)\geq\lm_>(q_i\cdot p_i)$, because we
     may drop all terms $s_i$ of $q_i$ with $\lm_>(f)<\lm_>(s_i\cdot
     p_i)$ and still retain the equality.
     Hence there are suitable $1\leq i_1<\ldots<i_l\leq k$ contributing to the leading term such that
     \begin{align*}
       \lt_>(f)&=\lt_>(q_{i_1}\cdot p_{i_1})+\ldots+\lt_>(q_{i_l}\cdot p_{i_l}) \\
       &=\lt_>(q_{i_1})\cdot p_{i_1}+\ldots+\lt_>(q_{i_l})\cdot p_{i_l} \in J.
     \end{align*}
     Moreover, this also means $f-\lt_>(f)\in J$ and we can continue this
     process leading term by leading term to see that every term of $f$
     lies in $J$.
   \end{proof}

   As an immediate consequence, we get the following.

   \begin{lemma}\label{lem:idealGeneratedByTerms1}
     Let $J\unlhd R[t,{\ux}]$ be an ideal generated by terms and let $>$ be a $t$-local monomial ordering on $\Mon(t,{\ux})$.
     Let $f\in R[t,{\ux}]$ such that $u\cdot f\in J$ for some $u\in R[t,{\ux}]$ with $\lt_>(u)=1$. Then $f\in J$.
   \end{lemma}
   \begin{proof}
     Lemma~\ref{lem:idealGeneratedByTerms0} implies $\lt_>(u\cdot f) =
     \lt_>(f) \in J$. Moreover, this implies $u\cdot\lt_>(f)\in J$ and
     therefore we also obtain
     \begin{displaymath}
       \lt_>(f-\lt_>(f)) = \lt_>(u\cdot (f-\lt_>(f)))= \lt_>(u\cdot f - u\cdot \lt_>(f)) \in J.
     \end{displaymath}
     We can again continue this process leading term by leading term to
     see that every term of $f$ lies in $J$. In particular, because $f$
     consists of only finitely many terms, $f\in J$.
   \end{proof}
   }

   \section{The Gr\"obner fan}\label{sec:groebnerfan}

   This section is devoted to the study of the Gr\"obner fan of an
   $\ux$-homoge\-neous ideal $I$ in $\Rtx$. We will show that it is a
   rational polyhedral fan (see~Theorem~\ref{thm:groebnerFan}),
   i.e.~it is a finite collection of rational
   polyhedral cones containing all faces of each cone in the
   collection and such that the intersection of each two cones in the
   collection is a face of both. For this we introduce the notion of
   an initially reduced standard basis of $I$ w.r.t.~a $t$-local
   monomial ordering, and show how such a
   standard basis can be used to read off the Gr\"obner cone $C_w(I)$
   (see~Proposition~\ref{prop:groebnerConeDecomposition}). All proofs
   in this section, except that of
   Proposition~\ref{prop:existenceinitiallyreducedstb}, are
   constructive, so that we end up with algorithms to compute
   Gr\"obner cones, provided that we can compute initially reduced
   standard bases.

   Let us first recall that a \emph{standard basis} of an ideal $I\unlhd\Rtx$ w.r.t.~a
   $t$-local monomial ordering $>$ is a finite subset $G$ of $I$, such
   the leading terms of its elements w.r.t.~$>$ generate the leading
   ideal of $I$. A standard basis of $I$ is automatically a generating
   set of $I$. A standard basis $G$ of $I$ is called \emph{reduced} if
   no term of the tail of any element of $G$ is in $\lt_>(I)$ and if it
   is minimal, i.e.~$\lt_>(I)$ cannot be generated by any proper
   subset of the set of leading terms of $G$. Observe that we forego
   any kind of normalisation of the leading
   coefficients that is normally done in polynomial rings over
   fields. By \cite[Alg.~4.2]{MRW15} reduced standard bases of
   $\ux$-homogeneous ideals in $\Rtx$ exist. However, even if the
   ideal $I$ is generated by polynomials in $\Rtxp$ the elements in a
   reduced standard of $I$ will in general be power series in $t$. We,
   therefore, now introduce a weaker notion.

   \begin{definition}[Initially reduced standard bases]\label{def:initialReduction}
     Let  $>$ be a $t$-local
     monomial ordering on $\Mon(t,\ux)$, and let
     $G,H \subseteq \Rtx$ be finite subsets where $G=\{
     g_1,\ldots,g_k\}$ with $g_i=\sum_{\alpha\in\NN^n} g_{i,\alpha}\cdot
     {\ux}^\alpha$, $g_{i,\alpha} \in \Rt$.
     \begin{enumerate}
     \item
       $G$ is \emph{reduced} w.r.t.~$H$, if no term
       of $\tail_>(g_i)$ lies in $\lt_>(H)$ for any $i$.
     \item
       We call $G$ \emph{initially
         reduced} w.r.t.~$H$, if the set
       \begin{displaymath}
         G':=\Big\{g_i':=\sum_{\alpha\in\NN} \lt_>(g_{i,\alpha})\cdot {\ux}^\alpha\;\bigmid\; i=1,\ldots,k\Big\},
       \end{displaymath}
       is reduced w.r.t.~$H$,
       i.e.~no term of $\tail_>(g_i')$ is in $\lt_>(H)$ for any $i$.
     \item
       We call a standard basis $G$ \emph{initially reduced}, if it is minimal and initially reduced w.r.t.~itself.
     \end{enumerate}
   \end{definition}

   \begin{example}
     Obviously, any reduced standard basis is initially reduced. The
     converse is false, since $G=\{1-t\}$ is initially reduced
     w.r.t.~any $t$-local monomial ordering, but it is not reduced.
   \end{example}

    \begin{proposition}[Existence of initially reduced standard bases]\label{prop:existenceinitiallyreducedstb}
     Any $\ux$-homogeneous ideal in $\Rtx$ has an initially reduced
     standard basis w.r.t.~any $t$-local monomial ordering.
   \end{proposition}
   \begin{proof}
     \lang{If $I\unlhd \Rtx$ is ${\ux}$-homogeneous, there exists a reduced
       standard basis $G$ w.r.t.~any $t$-local monomial ordering $>$ by
       \cite[Alg.~4.2]{MRW15} and $G$ is also
       initially reduced.}
     \kurz{By \cite[Alg.~4.2]{MRW15} reduced standard bases exist.}
   \end{proof}

   Algorithm~4.2 in \cite{MRW15} does not produce the basis $G$ in
   finite time, even if the input data is polynomial. The question,
   how to achieve this, is postponed to
   Section~\ref{sec:initiallyreduced}, where we treat a case
   of particular interest for the computation of tropical varieties
   over the $p$-adic numbers (see~\cite{MR15b}).  Instead we will now use
   initially reduced standard bases to give a constructive proof that the
   Gr\"obner fan of an ${\ux}$-homogeneous ideal indeed yields a
   polyhedral fan.

   \begin{lemma}\label{lem:initGroebner}
     Let  $G$ be an initially reduced standard basis of the
     $\ux$-homogeneous ideal $I\unlhd\Rtx$ w.r.t.~a $t$-local monomial
     ordering $>$.
     Then for all $w\in \RR_{<0}\times\RR^n$ we have
     \begin{displaymath}
       \initial_w(I)=\lt_>(I) \quad \Longleftrightarrow \quad \forall g\in G: \initial_w(g)=\lt_>(g).
     \end{displaymath}
   \end{lemma}
   \begin{proof}
     \begin{description}[leftmargin=0.5em,font=\normalfont]
     \item[$\Rightarrow$] Let $g\in G$. Then $\initial_w(g)\in\initial_w(I)=\lt_>(I)$.
       Writing $g=\sum_{\alpha\in\NN^n} g_\alpha\cdot {\ux}^\alpha$ with
       $g_\alpha\in\Rt$, note that the only terms of $g$ which can occur
       in $\initial_w(g)$ are of the form $\lt_>(g_\alpha)\cdot {\ux}^\alpha$
       for some $\alpha\in\NN^n$.
       And since our leading ideal is naturally generated by terms, these
       terms of $\initial_w(g)$ also lie in $\lt_>(I)$\kurz{.}\lang{ by Lemma~\ref{lem:idealGeneratedByTerms0}.}
       Because $G$ is initially reduced, we see that the only term of $g$
       which can occur in $\initial_w(g)$ is $\lt_>(g)$,
       i.e. $\initial_w(g)=\lt_>(g)$.
     \item[$\Leftarrow$] It is clear that $\initial_w(I)\supseteq
       \lt_>(I)$. For the converse, it suffices to show
       $\initial_w(f)\in\lt_>(I)$ for all $f\in I$.
       For that, consider the weighted ordering $>_w$ with weight vector
       $w$ and tiebreaker $>$, and note that $G$ is also a standard basis
       w.r.t.~that ordering. Hence any $f\in I$ will have a weak
       division with remainder $0$ w.r.t.~$G$ and $>_w$:
       \begin{displaymath}
         u\cdot f=q_1\cdot g_1+\ldots+q_k\cdot g_k.
       \end{displaymath}
       The weighted monomial ordering ensures, that there is no
       cancellation of highest weighted degree terms on the right hand
       side, and that $1$ is amongst the highest weighted degree terms in
       $u$. Taking the initial form w.r.t.~$w$ on both sides
       then yields:
       \begin{align*}
         \initial_w(u)\cdot\initial_w(f)
         &=\initial_w(q_{i_1})\cdot\initial_w(g_{i_1})+\ldots+\initial_w(q_{i_l})\cdot\initial_w(g_{i_l})\\
         &=\initial_w(q_{i_1})\cdot\lt_>(g_{i_1})+\ldots+\initial_w(q_{i_l})\cdot\lt_>(g_{i_l})
         \in \lt_>(I)
       \end{align*}
       for the $1\leq i_1<\ldots<i_l\leq k$ whose terms contribute to the
       highest weighted degree. Now since $\lt_>(I)$ is generated by
       terms, any term of $\initial_w(u)\cdot\initial_w(f)$ is contained
       in it. In particular, that means $\initial_w(f)\in\lt_>(I)$\kurz{.}
       \lang{by Lemma~\ref{lem:idealGeneratedByTerms1}.}
       \qedhere
     \end{description}
   \end{proof}

   \begin{example}\label{ex:31}
     Consider the ideal
     \lang{
     \begin{displaymath}
       \langle \underbrace{x-t^3x+t^3z-t^4z}_{=:g_1},\underbrace{y-t^3y+t^2z-t^4z}_{=:g_2}\rangle\unlhd \ZZ\llbracket t\rrbracket [x,y,z]
     \end{displaymath}}
     \kurz{
     \begin{displaymath}
       \langle g_1=x-t^3x+t^3z-t^4z,g_2=y-t^3y+t^2z-t^4z\rangle\unlhd \ZZ\llbracket t\rrbracket [x,y,z]
     \end{displaymath}}
     and the weighted ordering $>=>_v$ on $\Mon(t,x,y,z)$ with weight
     vector $v=(-1,3,3,3)\in\RR_{< 0}\times\RR^3$ and $t$-local
     lexicographical ordering $x>y>z>1>t$ as tiebreaker.

     Since $g_1$ and $g_2$ already form an initially reduced standard
     basis, the set whose Euclidean closure yields $C_>(I)$ is, due to
     Lemma~\ref{lem:initGroebner} given by
     \lang{\begin{align*}
       &\{w\in\RR_{< 0}\times\RR^3 \mid \initial_w(I)=\lt_>(I) \} = \\
       &\qquad \{w\in\RR_{< 0}\times\RR^3 \mid \initial_w(g_1)=\lt_>(g_1)=x \text{ and } \initial_w(g_2)=\lt_>(g_2)=y \}.
     \end{align*}}
   \kurz{\begin{displaymath}
        \{w\in\RR_{< 0}\times\RR^3 \mid \initial_w(g_1)=x, \initial_w(g_2)=y \}.
     \end{displaymath}}
     Hence it is cut out by the following two systems of inequalities:
     \begin{displaymath}
       \initial_w(g_1) = x \;\; \Longleftrightarrow \;\;
       \begin{cases}
         \deg_w(x)>\deg_w(t^3x) \\
         \deg_w(x)>\deg_w(t^3z) \\
         \deg_w(x)>\deg_w(t^4z)
       \end{cases}
       \Longleftrightarrow \;\;
       \begin{cases}
         0>w_0 & (\ast) \\
         w_1>3w_0+w_3 \\
         w_1>4w_0+w_3  & (\ast)
       \end{cases}
     \end{displaymath}
     and
     \begin{align*}
       \initial_w(g_2) = y \;\; \Longleftrightarrow \;\;
       \begin{cases}
         \deg_w(y)>\deg_w(t^3y) \\
         \deg_w(y)>\deg_w(t^2z) \\
         \deg_w(y)>\deg_w(t^4z)
       \end{cases}
       \Longleftrightarrow \;\;
       \begin{cases}
         0>w_0 & (\ast) \\
         w_2>2w_0+w_3 \\
         w_2>4w_0+w_3. & (\ast)
       \end{cases}
     \end{align*}

     The inequalities marked with $(\ast)$ are redundant, which is why
     the terms from which they arise are ignored in the definition of initial
     reducedness. Figure~\ref{fig:example31} shows an image in which we
     restrict ourselves to the affine subspace $\{w_0=-1, w_3=1\}$. Because
     the set is invariant under translation by $(0,1,1,1)$, no information is
     lost by doing so.

     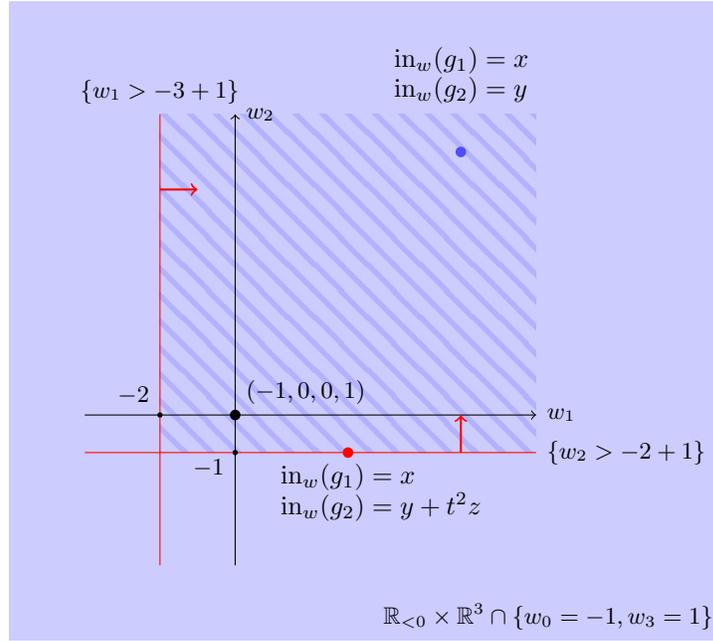
\begin{figure}[h]
       \centering
       \begin{tikzpicture}
         \fill[color=blue!20] (-3,-3) rectangle (6.5,5.5);
         \fill[pattern=my north west lines,pattern color=blue!30, line space=10pt,pattern line width=2pt] (-1,-0.5) rectangle (4,4);
         \draw[->] (-2,0) -- (4,0) node [anchor=west, font=\scriptsize] {$w_1$};
         \draw[->] (0,-2) -- (0,4) node [anchor=west, font=\scriptsize] {$w_2$};
         \fill (0,0) circle (2pt);
         \node[anchor=south west,font=\scriptsize] at (0,0) {$(-1,0,0,1)$};

         \draw[red] (-1,-2) -- (-1,4) node [black, font=\scriptsize, anchor=south] {$\{w_1>-3+1\}$};
         \draw[thick,red,->] (-1,3) -- (-0.5,3);
         \fill (-1,0) circle (1pt);
         \node[anchor=south east,font=\scriptsize] at (-1,0) {$-2$};

         \draw[red] (-2,-0.5) -- (4,-0.5) node [black, font=\scriptsize, anchor=west] {$\{w_2>-2+1\}$};
         \fill (0,-0.5) circle (1pt);
         \node[anchor=north east, font=\scriptsize, yshift=0.05cm] at (0,-0.5) {$-1$};
         \draw[thick,red,->] (3,-0.5) -- (3,0);

         \fill[red] (1.5,-0.5) circle (2pt);
         \node[anchor=north, font=\footnotesize] (expl) at (1.5,-0.5) {$\initial_w(g_1)=x$};
         \node[anchor=west, font=\footnotesize, yshift=-0.4cm] at (expl.west) {$\initial_w(g_2)=y+t^2z$};

         \fill[blue!70] (3,3.5) circle (2pt);
         \node [anchor=south, font=\footnotesize] (explTop) at (3,4) {$\initial_w(g_2)=y$};
         \node [anchor=west, font=\footnotesize, yshift=0.4cm] at (explTop.west) {$\initial_w(g_1)=x$};

         \node[anchor=south east,font=\scriptsize] at (6.5,-3) {$\RR_{< 0}\times\RR^3\cap\{w_0=-1,w_3=1\}$};
       \end{tikzpicture}
       \caption{$C_>(I)$ having the structure of a polyhedral cone}
       \label{fig:example31}
     \end{figure}

     Also note that while the weight vectors on the Euclidean boundary
     may not induce initial forms of $g_1$ and $g_2$ coinciding to the
     leading terms, the initial forms still contain the leading
     terms. \lang{This is a direct consequence of our last lemma.}
   \end{example}

   \begin{lemma}\label{lem:Cmembership}
     Let  $G$ be an initially reduced standard basis of the
     $\ux$-homogeneous ideal $I\unlhd\Rtx$ w.r.t.~a $t$-local monomial
     ordering $>$.
     Then for all $w\in
     \RR_{< 0}\times\RR^n$ we have
     \begin{displaymath}
       w\in C_>(I) \quad \Longleftrightarrow \quad \forall g\in G: \lt_>(\initial_w(g)) = \lt_>(g).
     \end{displaymath}
   \end{lemma}
   \begin{proof}
     Suppose $G=\{g_1,\ldots,g_k\}$. Similar to Example~\ref{ex:31},
     Lemma~\ref{lem:initGroebner} implies that the set $\{w\in \RR_{<
       0}\times\RR^n\mid \initial_w(I)=\lt_>(I)\}$ is cut out by a system
     of strict inequalities of the form:
     \lang{\begin{displaymath}
       \begin{array}{rcl}
         \deg_w(\lt_>(g_1)) &>& \deg_w(\tail_>(g_1)), \\
         \deg_w(\lt_>(g_2)) &>& \deg_w(\tail_>(g_2)), \\
         &\vdots& \\
         \deg_w(\lt_>(g_k)) &>& \deg_w(\tail_>(g_k)). \\
       \end{array}
     \end{displaymath}}
     \kurz{\begin{displaymath}
       \deg_w(\lt_>(g_i)) > \deg_w(\tail_>(g_i)),\;\;\; i=1,\ldots,k.
     \end{displaymath}}
     Note that each line, despite $g_i\in\Rtx$, only yields a finite
     amount of minimal inequalities, since higher degrees of $t$ yield
     redundant inequalities.
     Therefore, its Euclidean closure $C_>(I)$ is given by a system of inequalities of the form
     \lang{\begin{displaymath}
       \begin{array}{rcl}
         \deg_w(\lt_>(g_1)) &\geq& \deg_w(\tail_>(g_1)) \\
         \deg_w(\lt_>(g_2)) &\geq& \deg_w(\tail_>(g_2)) \\
         &\vdots& \\
         \deg_w(\lt_>(g_k)) &\geq& \deg_w(\tail_>(g_k)) \\
       \end{array}
     \end{displaymath}}
     \kurz{\begin{displaymath}
       \deg_w(\lt_>(g_i)) \geq \deg_w(\tail_>(g_i)),\;\;\; i=1,\ldots,k,
     \end{displaymath}}
     which is equivalent to $\lt_>(g_i)$ occuring in $\initial_w(g)$ and translates to the condition in the claim.
   \end{proof}

   We can use this result to generalise the statement of Lemma~\ref{lem:initGroebner}
   to weight vectors in the boundary of $C_>(I)$.

   \begin{lemma}\label{lem:LeadOfInitialIdeal}
     Let $>$ be a $ t$-local monomial ordering and let $I\unlhd\Rtx$
     be an $\ux$-homogeneous ideal. Then for all $w\in C_>(I)$, $w\in \RR_{<0}\times\RR^n$, we have
     \begin{displaymath}
       \lt_>(\initial_w(I)) = \lt_>(I).
     \end{displaymath}
   \end{lemma}
   \begin{proof} Let $G$ be an initially reduced standard basis of $I$
     w.r.t.~$>$.
     Since $\lt_>(\initial_w(g))=\lt_>(g)$ for all $g\in G$ by Lemma~\ref{lem:Cmembership},
     we have
     \begin{displaymath}
       \lt_>(I) = \langle \lt_>(g) \mid g\in G \rangle
       \underset{\ref{lem:Cmembership}}{\overset{\text{Lem.}}{=}}
       \langle \lt_>(\initial_w(g)) \mid g\in G \rangle
       \subseteq \lt_>(\initial_w(I)).
     \end{displaymath}
     For the opposite inclusion, we can again consider the weighted
     ordering $>_w$. Given any $h\in \initial_w(I)$ with $h=\initial_w(f)$ for
     some $f\in I$, this $f$ has a weak division with remainder $0$
     w.r.t.~$G=\{g_1,\ldots,g_k\}$ under $>_w$:
     \begin{displaymath}
       u\cdot f=q_1\cdot g_1+\ldots+q_k\cdot g_k.
     \end{displaymath}
     Because no cancellation of highest weighted degree terms occurs on
     the right, taking the initial forms on both sides yields:
     \begin{displaymath}
       \initial_w(u)\cdot\initial_w(f)
       =\initial_w(q_{i_1})\cdot\initial_w(g_{i_1})+\ldots+\initial_w(q_{i_l})\cdot\initial_w(g_{i_l})
     \end{displaymath}
     for the $1\leq i_1<\ldots<i_l\leq k$ whose terms contribute to the
     highest weighted degree. Moreover,
     $\lt_>(\initial_w(u))=\lt_{>_w}(u)=1$. Therefore taking the leading
     terms on both sides produces:
     \begin{align*}
       \lt_>(\initial_w(f))
       &\underset{\phantom{\ref{lem:Cmembership}}}{\overset{\phantom{\text{Lem.}}}{=}}
       q_{i_1}'\cdot\lt_>(\initial_w(g_{i_1}))+\ldots
       +q_{i_l}'\cdot\lt_>(\initial_w(g_{i_l})) \\
       &\underset{\ref{lem:Cmembership}}{\overset{\text{Lem.}}{=}}
       q_{i_1}'\cdot\lt_>(g_{i_1})+\ldots
       +q_{i_l}'\cdot\lt_>(g_{i_l}) \in \lt_>(I),
     \end{align*}
     where we abbreviated $q_{i_j}':=\lt_>(\initial_w(q_{i_j}))$ for $j=1,\ldots,l$.
   \end{proof}

   Combining the previous lemmata we deduce how initially reduced
   standard bases of restrict to initially reduced standard bases
   of initial ideals.

   \begin{proposition}\label{prop:BasisOfInitialIdeal}
     Let  $G$ be an initially reduced standard basis of the
     $\ux$-homogeneous ideal $I\unlhd\Rtx$ w.r.t.~a $t$-local monomial
     ordering $>$.
     Then for all $w\in C_>(I)$ with $w_0<0$ the set
     \begin{displaymath}
       H:=\{\initial_w(g)\mid g\in G\}
     \end{displaymath}
     is an initially reduced standard basis of $\initial_w(I)$ w.r.t.~the same ordering.
   \end{proposition}
   \begin{proof}
     By the previous Lemmata, we have
     \begin{displaymath}
       \lt_>(\initial_w(I)) \underset{\ref{lem:LeadOfInitialIdeal}}{\overset{\text{Lem.}}{=}} \lt_>(I)
       = \langle\lt_>(g)\mid g\in G\rangle
       \underset{\ref{lem:Cmembership}}{\overset{\text{Lem.}}{=}}
       \langle\lt_>(\initial_w(g))\mid g\in G \rangle,
     \end{displaymath}
     and therefore $H$ is a standard basis of $\initial_w(I)$.
     Moreover, because $G$ was initially reduced, so is $H$.
   \end{proof}

   \begin{example}\label{ex:32}
     Given the same ideal and ordering as in Example~\ref{ex:31}, $g_1\lang{=
     x-t^3x+t^3z-t^4z}$ and $g_2\lang{=y-t^3y+t^2z-t^4z}$ form an initially
     reduced standard basis.
     Because $w:=(-1,2,-1,1)\in C_>(I)$,
     Proposition~\ref{prop:BasisOfInitialIdeal} implies that the
     initial ideal $\initial_w(I)$ has the initially reduced standard bases
     \lang{\begin{displaymath}
       \{\initial_w(g_1), \initial_w(g_2)\} = \{ x, y+t^2z \}.
     \end{displaymath}}
     \kurz{$ \{\initial_w(g_1), \initial_w(g_2)\} = \{ x, y+t^2z \}$.}
     As we go over all weight vectors in $C_>(I)$ in the affine subspace,
     we obtain four distinct initial ideals as illustrated in Figure~\ref{fig:example32}.
     \begin{figure}[h]
       \centering
       \begin{tikzpicture}
         \fill[color=blue!20] (-4.5,-3) rectangle (5,5);
         \fill[pattern=my north west lines,pattern color=blue!30, line space=10pt,pattern line width=2pt] (-1,-0.5) rectangle (4,4);
         \draw[->] (-2,0) -- (4,0) node [anchor=west, font=\scriptsize] {$w_1$};
         \draw[->] (0,-2) -- (0,4) node [anchor=west, font=\scriptsize] {$w_2$};

         \draw[red] (-1,-0.5) -- (-1,4);
         \fill (-1,0) circle (1pt);
         \node[anchor=south east,font=\scriptsize] at (-1,0) {$-2$};

         \draw[red] (-1,-0.5) -- (4,-0.5);
         \fill (0,-0.5) circle (1pt);
         \node[anchor=north east, font=\scriptsize, yshift=0.05cm] at (0,-0.5) {$-1$};

         \fill[red] (-1,-0.5) circle (2pt);

         \node[font=\small] at (1.5,1.5) {$\{x,y\}$};
         \node[anchor=north, font=\small] at (1.5,-0.5) {$\{x,y+t^2z\}$};
         \node[anchor=east, font=\small] at (-1,1.5) {$\{x+t^3z,y\}$};
         \node[anchor=north east, font=\small] at (-1,-0.5) {$\{x+t^3z,y+t^2z\}$};

         \fill (1,-0.5) circle (2pt);
         \node[anchor=south, font=\scriptsize] at (1,-0.5) {$(-1,2,-1,1)$};

         \node[anchor=south east,font=\scriptsize] at (5,-3) {$\RR_{< 0}\times\RR^3\cap\{w_0=-1,w_3=1\}$};
       \end{tikzpicture}
       \caption{standard bases of initial ideals with various weights}
       \label{fig:example32}
     \end{figure}
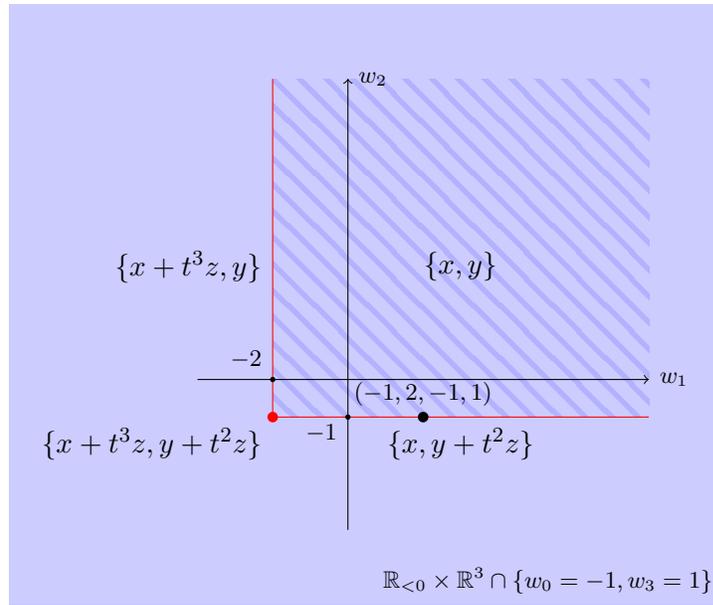
   \end{example}

   \lang{The finiteness of distinct initial ideals holds in general and it can be stated as an easy corollary.}

   \begin{corollary}\label{cor:finiteInitialIdeals}
     Any $\ux$-homogeneous ideal $I\unlhd\Rtx$ has only finitely many
     distinct initial ideals. In particular, $I$ has  only finitely
     many Gr\"obner cones.
   \end{corollary}
   \begin{proof}
     Note first that due to Lemma~\ref{lem:Cmembership} every weight
     vector $w$ is contained in $C_>(I)$ for some $t$-local monomial
     ordering $>$, just choose any refinement $>_w$ of the partial
     ordering induced by $w$. By
     Proposition~\ref{prop:BasisOfInitialIdeal} the initial ideal
     $\initial_w(I)$ is determined by any initially reduced standard
     basis of $I$ w.r.t.~$>$. Since by
     Proposition~\ref{prop:finiteLeadingIdeals} there are only
     finitely many distinct $C_>(I)$, it suffices to argue why a fixed
     $C_>(I)$ can only lead to finitely many distinct initial ideals
     $\initial_w(I)$, since this implies that there are only finitely
     many $C_w(I)$ and hence only finitely many $C_w^0(I)$.

     To this end note that an arbitrary element $g=\sum_{\alpha\in\NN^n} g_\alpha
     {\ux}^\alpha\in\Rtx$ with $g_\alpha\in\Rt$ has only finitely many
     distinct initial forms.
     Consider a weight vector $w\in\RR_{<
       0}\times\RR^n$, and let $>$ be a $t$-local monomial ordering. The
     initial forms of $g$ w.r.t.~$>$ are of the form
     \begin{displaymath}
       \initial_w(g) = \sum_{\alpha\in\Lambda} \lt_>(g_\alpha) \cdot {\ux}^\alpha
     \end{displaymath}
     for a finite set $\Lambda\subseteq\{\alpha\in\NN^n\mid
     g_\alpha\neq 0\}$.
     Thus a fixed initially reduced standard basis of $I$ w.r.t.~$>$
     admits by Proposition~\ref{prop:BasisOfInitialIdeal} only
     finitely many choices for generating sets of initial ideals
     $\initial_w(I)$ and hence only finitely many initial ideals.
   \end{proof}

   The next proposition allows us to read off the inequalities and
   equations of the Gr\"obner cones, from which we can derive the
   remaining properties needed to show that they form a polyhedral fan.

   \begin{proposition}\label{prop:initialEquality}
     Let  $G$ be an initially reduced standard basis of the
     $\ux$-homogeneous ideal $I\unlhd\Rtx$ w.r.t.~a $t$-local monomial
     ordering $>$ and let $w\in C_>(I)$ with $w_0<0$.
     Then for all $v\in \RR_{< 0}\times\RR^n$
     we have
     \begin{displaymath}
       \initial_v(I)=\initial_w(I) \quad \Longleftrightarrow \quad \forall g\in G: \initial_v(g)=\initial_w(g).
     \end{displaymath}
   \end{proposition}
   \begin{proof}
     \begin{description}[leftmargin=0.5em,font=\normalfont]
     \item[$\Leftarrow$] For $g\in G$ note that
       \begin{displaymath}
         \lt_>(\initial_v(g))=\lt_>(\initial_w(g))
         \underset{\ref{lem:Cmembership}}{\overset{\text{Lem.}}{=}} \lt_>(g),
       \end{displaymath}
       thus $v\in C_>(I)$, again by Lemma~\ref{lem:Cmembership}. This
       allows us to use Proposition~\ref{prop:BasisOfInitialIdeal}, which
       says that $\initial_w(I)$ and $\initial_v(I)$ share a common
       standard basis, therefore they must coincide.
     \item[$\Rightarrow$] Let $g\in G$. On the one hand, Lemma~\ref{lem:Cmembership} implies that $\lt_>(g)$ is a term of
       $\initial_w(g)$. On the other hand,
       \begin{displaymath}
         \lt_>(\initial_v(g)) \in \lt_>(\initial_v(I))=\lt_>(\initial_w(I))
         \underset{\ref{lem:LeadOfInitialIdeal}}{\overset{\text{Lem.}}{=}} \lt_>(I).
       \end{displaymath}
       But because $G$ is initially reduced, the only term of $g$
       occurring in $\initial_v(g)$ and $\lt_>(I)$ is $\lt_>(g)$. Thus
       $\lt_>(g)$ is also a term of $\initial_v(g)$.

       Now consider $\initial_w(g)-\initial_v(g) \in
       \initial_w(I)=\initial_v(I)$. Our previous arguments show that
       $\lt_>(\initial_w(g)-\initial_v(g))\neq\lt_>(g)$. However, because

       \begin{displaymath}
         \lt_>(\initial_w(g)-\initial_v(g))\in\lt_>(\initial_w(I))
         \underset{\ref{lem:LeadOfInitialIdeal}}{\overset{\text{Lem.}}{=}} \lt_>(I),
       \end{displaymath}
       it is another term of $\initial_w(g)$ or
       $\initial_v(g)$ in $\lt_>(I)$, which must be $0$. \qedhere
     \end{description}
   \end{proof}

   \begin{example}\label{ex:33}
     Consider the same ideal and ordering as in Example~\ref{ex:31} and
     Example~\ref{ex:32}, where $g_1 = x-t^3x+t^3z-t^4z$ and
     $g_2=y-t^3y+t^2z-t^4z$ form an initially reduced standard basis.

     For $w=(-1,2,-1,1)\in C_>(I)$ we have by Proposition~\ref{prop:initialEquality}:
     \begin{displaymath}
       \initial_{w'}(I)=\initial_w(I)=\langle x,y+t^2z\rangle \quad \Longleftrightarrow
       \quad \begin{cases} \initial_{w'}(g_1)=x, \\
         \initial_{w'}(g_2)=y+t^2z. \end{cases}
     \end{displaymath}
     Therefore, its equivalence class of weight vectors $w'\in\RR_{<
       0}\times\RR^3$ such that $\initial_{w'}(I)=\initial_w(I)$ is
     determined by the following system of inequalities and equations:
     \begin{displaymath}
       \initial_{w'}(g_1) = x \;\;\; \Longleftrightarrow \;\;\;
       \begin{cases}
         \deg_{w'}(x)>\deg_{w'}(t^3x) \\
         \deg_{w'}(x)>\deg_{w'}(t^3z) \\
         \deg_{w'}(x)>\deg_{w'}(t^4z)
       \end{cases}
       \Longleftrightarrow \;\;\;
       \begin{cases}
         0>w_0' \\
         w_1'>3w_0'+w_3' \\
         w_1'>4w_0'+w_3'
       \end{cases}
     \end{displaymath}
     \begin{displaymath}
       \initial_w(g_2) = y+t^2z \;\; \Longleftrightarrow \;\;
       \begin{cases}
         \deg_{w'}(y)>\deg_{w'}(t^3y) \\
         \deg_{w'}(y)=\deg_{w'}(t^2z) \\
         \deg_{w'}(y)>\deg_{w'}(t^4z)
       \end{cases}
       \Longleftrightarrow \;\;
       \begin{cases}
         0>w_0' \\
         w_2'=2w_0'+w_3' \\
         w_2'>4w_0'+w_3'
       \end{cases}
     \end{displaymath}
     In particular, its Euclidean closure, the Gr\"obner cone $C_w(I)$,
     is the face of $C_>(I)$ cut out by the hyperplane $\{
     w_2'=2w_0'+w_3' \}$.

     In fact, Proposition~\ref{prop:initialEquality} implies that
     $C_>(I)$ is stratified by equivalence classes of weight vectors as
     Figure~\ref{fig:example32} already suggested. Each  class
     is an open polyhedral cone whose Euclidean closure yields a face of
     $C_>(I)$.






   \end{example}

   \begin{proposition}\label{prop:groebnerConeDecomposition}
     For any $\ux$-homogeneous ideal $I\unlhd\Rtx$ and for any $w\in
     \RR_{< 0}\times\RR^n$, the Gr\"obner cones $C_w(I)$ and
     $C_w^0(I)$ are closed rational polyhedral cones.
   \end{proposition}
   \begin{proof}
     Let $>$ be a $t$-local weighted monomial ordering w.r.t.~a weight vector $w$, and let
     $G$ be an initially reduced standard basis of $I$ w.r.t.~$>$.

     Suppose $G=\{g_1,\ldots,g_k\}$ with $g_i=\sum_{\beta,\alpha}
     c_{\alpha,\beta,i} \cdot t^\beta {\ux}^\alpha$. Let $\Lambda_i$ be the
     finite set of exponent vectors with minimal entry in $t$,
     \begin{displaymath}
       \Lambda_i:=\left\{(\beta,\alpha)\in\NN\times\NN^n\suchthat \alpha \in \NN^n, \beta = \min \{\beta'\in\NN\mid c_{\alpha,\beta',i}\neq 0\} \right\}.
     \end{displaymath}
     Similar to Example~\ref{ex:33}, Proposition~\ref{prop:initialEquality} implies that the equivalence class of
     $w$, $\{v\in \RR_{< 0}\times\RR^n \mid
     \initial_v(I)=\initial_w(I)\}$, is cut out by a system of
     inequalities and equations
     \begin{align*}
       v\cdot (\beta,\alpha) &> v\cdot (\delta,\gamma), \quad \text{ for all } (\beta,\alpha), (\delta,\gamma)\in \Lambda_i \text{ with } w\cdot (\beta,\alpha)> w\cdot (\delta,\gamma), \\
       v\cdot (\beta,\alpha) &= v\cdot (\delta,\gamma), \quad \text{ for all } (\beta,\alpha), (\delta,\gamma)\in \Lambda_i \text{ with } w\cdot (\beta,\alpha)= w\cdot (\delta,\gamma).
     \end{align*}
     Therefore, the equivalence class forms a relative open polyhedral
     cone contained in the open lower half space $\RR_{<0}\times\RR^n$
     and
     its closure $C_w(I)$ yields a closed polyhedral cone in the closed
     lower half space $\RR_{\leq 0}\times\RR^n$. In particular, $C_w^0(I)
     = C_w(I)\cap (\{0\}\times\RR^n)$ is also a closed polyhedral cone.
   \end{proof}

   \begin{corollary}\label{cor:facesOfGroebnerCones}
     Let $I\unlhd\Rtx$ be an $\ux$-homogeneous ideal and let $w\in
     \RR_{<0}\times\RR^n$. Then any face $\tau\leq C_w(I)$ with
     $\tau\nsubseteq\{0\}\times\RR^n$ coincides with the closure of the
     equivalence class of any weight vector in its relative interior.

     In particular, each face $\tau\leq C_w(I)$ is a Gr\"obner cone of the form
     $\tau=C_v(I)$ or $\tau=C_v^0(I)$ for some $v\in C_w(I)$ and each
     face $\tau\leq C_w^0(I)$ is a Gr\"obner cone of the form $\tau=C_v^0(I)$ for some
     $v\in C_w(I)$.
   \end{corollary}
   \begin{proof}
     Consider again the system of inequalities and equations that cut out $C_w(I)$ in
     the proof of the previous
     Proposition~\ref{prop:groebnerConeDecomposition}, which we
     obtained from the sets of
     exponent vectors $\Lambda_1,\ldots,\Lambda_k$ of an initially reduced standard basis w.r.t.~a weighted ordering $>_w$.

     A face $\tau\leq C_w(I)$ is cut out by supporting hyperplanes, on
     which some of the inequalities above become equations. Assuming that
     $\tau\nsubseteq\{0\}\times\RR^n$, all weight
     vectors in the relative interior yield the same initial forms on
     $g_1,\ldots,g_k\in G$, since they satisfy the same equations and
     inequalities on the exponent vectors $\Lambda_1,\ldots,\Lambda_k$. This
     implies that they belong to the same equivalence class whose closure
     is then $\tau$. In particular, $\tau=C_v(I)$.

     And any face $\tau\leq C_w^0(I)\leq C_w(I)$ can be cut out by a
     supporting hyperplane which also cuts out a face $C_v(I)\leq
     C_w(I)$. It is then clear that $\tau=C_v^0(I)$.
   \end{proof}

   \begin{proposition}\label{prop:groebnerConeIntersection}
     Let $I\unlhd\Rtx$ be an $\ux$-homogeneous ideal and let $C_u(I)$
     and $C_v(I)$ be two interior Gr\"obner cones of $I$ such that
     $C_u(I)\cap C_v(I)\nsubseteq\{0\}\times\RR^n$.
     Then $C_u(I)\cap C_v(I)$ is an interior Gr\"obner cone and it is a
     face of both\lang{ $C_u(I)$ and $C_v(I)$}.
   \end{proposition}
   \begin{proof}
     By Proposition~\ref{prop:groebnerConeDecomposition}, both
     $C_u(I)\cap(\RR_{<0}\times\RR^n)$ and
     $C_v(I)\cap(\RR_{<0}\times\RR^n)$ can be decomposed
     into a union of equivalence classes, and hence so can $(C_u(I)\cap C_v(I))\cap(\RR_{<0}\times\RR^n)\neq\emptyset$.

     Let $k:=\dim(C_u(I)\cap C_v(I))$. Then the intersection contains exactly one equivalence class of dimension $k$:
     If there were none, then the intersection would be covered by a collection of lower dimensional
     open cones of which there are, however, only finitely many by Corollary~\ref{cor:finiteInitialIdeals}.
     If there were more than one, then that would contradict Proposition~\ref{prop:groebnerConeDecomposition}, which states that the closure of each equivalence class
     yields a distinct face of both $C_u(I)$ and $C_v(I)$, and no two $k$-dimensional faces of a polyhedral cone may
     be cut out by the same $k$-dimensional supporting hyperplane.

     So let $w$ be in the maximal equivalence class in $C_u(I)\cap C_v(I)$.
     Taking the Euclidean closure, we necessarily have $C_w(I)=C_u(I)\cap C_v(I)$,
     and, by Corollary~\ref{prop:groebnerConeDecomposition}, it is a face of both $C_u(I)$ and $C_v(I)$.
   \end{proof}

   Note that the proposition above falls a bit short in proving that the
   intersection of two Gr\"obner cones yields a face of both, as it only
   covers Gr\"obner cones with an intersection in the open part of the
   lower halfspace.
   To cover the remaining intersection, we need  some results on recession fans.

   \begin{definition}
     Let $I\unlhd\Rtx$ be an $\ux$-homogeneous ideal and $w\in \RR_{<0}\times\RR^n$.
     For an interior Gr\"obner cone $C_w(I)$ let $C_w^{-1}(I)$ denote the intersection
     \begin{displaymath}
       C_w^{-1}(I) := C_w(I) \cap (\{-1\}\times\RR^n).
     \end{displaymath}
     It is a polytope whose \emph{recession cone} is defined to be the
     set of all weight vectors in $\RR_{\leq 0}\times\RR^n$ under whose
     translation it is closed,
     \begin{displaymath}
       \rec(C_w^{-1}(I)) := \{ v\in \RR_{\leq 0}\times\RR^n\mid v+C_w^{-1}(I)\subseteq C_w^{-1}(I) \}.
     \end{displaymath}
     Note that $C_w^{-1}(I)\subseteq \{-1\}\times\RR^n$ necessarily implies $\rec(C_w^{-1}(I))\subseteq \{0\}\times\RR^n$.
   \end{definition}

   \begin{proposition}\label{cor:recessionFan}
     Let $I\unlhd\Rtx$ be an $\ux$-homogeneous ideal.
     \begin{enumerate}[leftmargin=*]
     \item The collection \kurz{$\{C_w^{-1}(I)\mid w\in\RR_{<0}\times\RR^n\}$}
       \lang{\begin{displaymath}
         \{C\cap(\{-1\}\times\RR^n)\;|\;C\in \Sigma(I)\}=\{C_w^{-1}(I)\mid w\in\RR_{<0}\times\RR^n\}
       \end{displaymath}}
       is a polyhedral complex whose support is the affine hyperplane $\{-1\}\times\RR^n$.
     \item For any weight vector $w\in \RR_{<0}\times\RR^n$, $C_w^0(I)=\rec(C_w^{-1}(I))$.
     \item The collection \kurz{$\{C_w^0(I)\mid w\in\RR_{<0}\times\RR^n\}$}
       \lang{\begin{displaymath}
         \{C\cap(\{0\}\times\RR^n)\;|\;C\in \Sigma(I)\}=
         \{C_w^0(I)\mid w\in\RR_{<0}\times\RR^n\}
       \end{displaymath}}
       is a polyhedral fan whose support is the boundary hyperplane $\{0\}\times\RR^n$.
     \end{enumerate}
   \end{proposition}
   \begin{proof}
     1.~follows from Proposition~\ref{prop:groebnerConeDecomposition},
     Corollary~\ref{cor:facesOfGroebnerCones} and
     Proposition~\ref{prop:groebnerConeIntersection}.
     2.~is clear, and 3.~follows
     from \cite[Cor.~3.10]{BS11}.
   \end{proof}

   We can now supplement the missing intersections in Proposition~\ref{prop:groebnerConeIntersection}.

   \begin{corollary}\label{cor:groebnerConeIntersection}
     Let $I\unlhd\Rtx$ be an $\ux$-homogeneous ideal and let
     $u,v\in\RR_{<0}\times\RR^n$. Then the intersections
     $C_u^0(I)\cap C_v^0(I)$, $C_u^0(I)\cap C_v(I)$ are boundary Gr\"obner cones of $I$ and
     they are faces of the intersected cones.
   \end{corollary}
   \begin{proof}
     Since the boundary Gr\"obner cones form a polyhedral fan by
     Proposition~\ref{cor:recessionFan}, the intersection $C_u^0(I)\cap
     C_v^0(I)$ is a face of both. In particular, by Corollary~\ref{cor:facesOfGroebnerCones}, there is a weight vector
     $w\in\RR_{<0}\times\RR^n$ with
     \begin{displaymath}
       C_u^0(I)\cap C_v^0(I) = C_w^0(I).
     \end{displaymath}
     And for the intersection of a boundary Gr\"obner cone and an interior Gr\"obner cone, note that
     \begin{displaymath}
       C_u^0(I)\cap C_v(I)=C_u^0(I)\cap C_v^0(I) = C_w^0(I). \qedhere
     \end{displaymath}
   \end{proof}

   We are now able to prove the main theoretical result of the paper.

   \begin{theorem}\label{thm:groebnerFan}
     Let $I\unlhd \Rtx$ be an ${\ux}$-homogeneous ideal, then the Gr\"obner fan
     \begin{displaymath}
       \Sigma(I) = \{ C_w(I)\mid w\in \RR_{<0}\times\RR^n \} \cup \{ C_w^0(I)\mid w\in \RR_{<0}\times\RR^n \}
     \end{displaymath}
     is a rational polyhedral fan with support $\RR_{\leq 0}\times\RR^n$.
   \end{theorem}
   \begin{proof}
     Proposition~\ref{prop:groebnerConeDecomposition} shows that each
     Gr\"obner cone is a polyhedral cone, while Corollary~\ref{cor:facesOfGroebnerCones} proves that each face of a Gr\"obner
     cone is again a Gr\"obner cone. Proposition~\ref{prop:groebnerConeIntersection} and Corollary~\ref{cor:groebnerConeIntersection} infer that the intersection of
     two Gr\"obner cones is a face of each, and Corollary~\ref{cor:finiteInitialIdeals} shows that there are only finitely
     many of them.
   \end{proof}

   \begin{example}
     Consider the following ideal generated by polynomials
     \begin{displaymath}
       \langle 2x+2y, t+2 \rangle \unlhd \ZZ\llbracket t\rrbracket [x,y].
     \end{displaymath}
     Now because the ideal is generated by elements in $\ZZ[t,x,y]$, one
     might be tempted to believe that restricting ourselves to the
     polynomial ideal
     \begin{displaymath}
       \langle 2x+2y, t+2 \rangle \unlhd \ZZ [t,x,y],
     \end{displaymath}
     might allow us to work with weight vectors $\RR_{\geq 0}\times\RR^2$
     with positive weight in $t$, obtain similar results about the
     existence of a Gr\"obner fan there and patch the two Gr\"obner fans
     in $\RR_{\leq 0}\times\RR^2$ and in $\RR_{\geq 0}\times\RR^2$
     together.

     While the existence of a Gr\"obner fan in the positive halfspace is
     true for our specific example, note that the two Gr\"obner fans
     cannot be glued together to a polyhedral fan on $\RR\times\RR^2$, as
     illustrated in Figure~\ref{fig:groebnerFanNoGlobal}.
     \begin{figure}[h]
       \centering
       \begin{tikzpicture}
         \fill [blue!20] (-4,0) rectangle (4,3);
         \fill [red!20] (-4,0) rectangle (4,-3);
         \node [font=\small] at (0,-1.5) {$\langle 2 \rangle$};
         \node [font=\small] at (-2,1.5) {$\langle 2x,t \rangle$};
         \node [font=\small] at (2,1.5) {$\langle 2y,t \rangle$};
         \node [anchor=south, font=\small] at (0,3) {$\langle 2x+2y,t \rangle$};
         \draw [line width=0.15cm, blue!20] (-4,0) -- (4,0);
         \draw [line width=0.15cm, red!20, dashed] (-4,0) -- (4,0);
         \draw (-4,0) -- (4,0);
         \draw (0,0) -- (0,3);
         \draw[->, shorten <=5pt, shorten >=5pt] (-5,0) -- (-4,0);
         \node [anchor=east]  at (-5,0) {??};
         \draw[->, shorten <=5pt, shorten >=5pt] (5,0) -- (4,0);
         \node [anchor=west]  at (5,0) {??};
         \node [anchor=north east, font=\small] at (4,-3) {$\RR\times\RR^2\cap \{ w_x=w_y\}$};
       \end{tikzpicture}
       \caption{$\Sigma(\langle 2x+2y, t+2 \rangle)$ on $\RR\times\RR^2$...?}
       \label{fig:groebnerFanNoGlobal}
     \end{figure}
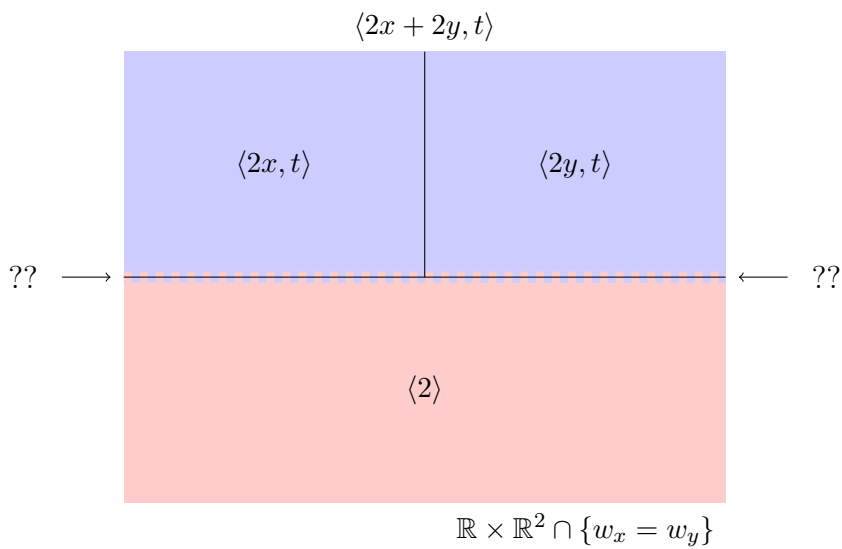
   \end{example}

   As demonstrated in Example~\ref{ex:33} and used in the proof of
   Proposition~\ref{prop:groebnerConeDecomposition}, Proposition~\ref{prop:initialEquality} allows us to read of the inequalities and
   equations of a Gr\"obner cone from an initially reduced standard
   basis. This can be done as described in the following algorithm.

   \begin{algorithm}[Inequalities and equations of a Gr\"obner cone]\label{alg:groebnerConeNoWeight}\  \vspace*{-3ex}
     \begin{algorithmic}[1]
       \REQUIRE{$(H,G,>)$, where for an ${\ux}$-homogeneous ideal $I$ and an undetermined weight vector $w\in\RR_{< 0}\times\RR^n$
         \begin{enumerate}[leftmargin=*, itemsep=0pt]
         \item $>$ a $t$-local monomial ordering such that $w\in C_>(I)$,
         \item $G=\{g_1,\ldots,g_k\}$ an initially reduced standard basis of $I$ w.r.t. $>$,
         \item $H=\{h_1,\ldots,h_k\}$ with $h_i=\initial_w(g_i)$.
         \end{enumerate}}
       \ENSURE{$(A,B)$, a pair of matrices
         \begin{displaymath}
           A\in\Mat(l_A\times(n+1),\RR), \quad B\in\Mat(l_B\times(n+1),\RR)
         \end{displaymath}
         such that
         \begin{displaymath}
           C_w(I)=\{v\in\RR_{\leq 0}\times\RR^n\mid A\cdot v\in(\RR_{\geq 0})^{l_A} \text{ and } B\cdot v = 0\in\RR^{l_B} \}.
         \end{displaymath}}
       \FOR{$i=1,\ldots,k$}
       \STATE Suppose $g_i=\sum_{\beta,\alpha} c_{\alpha,\beta,i} \cdot t^\beta {\ux}^\alpha$ and $\lm_>(g)=t^{\delta}{\ux}^{\gamma}$.
       \STATE Construct the set of exponent vectors with minimal entry in $t$,
       \begin{displaymath}
         \Lambda_i:=\left\{(\beta,\alpha)\in\NN\times\NN^n\suchthat \alpha \in \NN^n, \beta = \min \{\beta'\in\NN\mid c_{\alpha,\beta',i}\neq 0\} \right\}.
       \end{displaymath}
       \STATE Construct a set of vectors that will yield the inequalities,
       \begin{displaymath}
         \Omega_i:=\left\{ (\delta,\gamma)-(\alpha,\beta)\in \RR\times\RR^n \mid (\alpha,\beta)\in \Lambda_i,\;\; (\alpha,\beta) \neq (\delta,\gamma)\right\}.
       \end{displaymath}
       \ENDFOR
       \STATE Let $A$ be a matrix whose row vectors consist of $\bigcup_{i=1}^k \Omega_i$.
       \FOR{$i=1,\ldots,k$}
       \STATE Suppose $h_i=\sum_{\beta,\alpha} d_{\alpha,\beta,i} \cdot t^\beta {\ux}^\alpha$.
       \STATE Construct the set of exponent vectors with minimal entry in $t$,
       \begin{displaymath}
         \Lambda_i':=\left\{(\beta,\alpha)\in\NN\times\NN^n\suchthat \alpha \in \NN^n, \beta = \min \{\beta'\in\NN\mid d_{\alpha,\beta',i}\neq 0\} \right\}.
       \end{displaymath}
       \STATE Construct a set of vectors that will yield the equations,
       \begin{displaymath}
         \Theta_i:=\left\{ a-b\in \RR\times\RR^n \mid a,b\in \Lambda_i' \right\}.
       \end{displaymath}
       \ENDFOR
       \STATE Let $B$ be a matrix whose row vectors consist of $\bigcup_{i=1}^k \Theta_i$.
       \RETURN{$(A,B)$.}
     \end{algorithmic}
   \end{algorithm}



   We close the section with an example which shows why it is
   important that the standard basis is initially reduced in order to
   determine the corresponding Gr\"obner cone. It is an example abiding
   to the special assumptions on $R$ and $I$ considered in
   Section~\ref{sec:initiallyreduced} (see Page~\pageref{page:spezialannahme}).

   \begin{example}\label{ex:spezialannahmen}
     Let $\Rtx=\ZZ\llbracket t \rrbracket[x,y,z]$ and let $>=>_v$ be a
     weighted ordering with weight vector
     $v=(-1,1,1,1)\in\RR_{<0}\times\RR^3$ and the $t$-local
     lexicographical ordering $x>y>1>t$ as tiebreaker. We consider the
     ideal
     \begin{displaymath}
       I =  \langle g_0= 2-t, g_1=x+t^2y+t^3z, g_2=y+tx+t^2z\rangle \unlhd \ZZ\llbracket t\rrbracket[x,y,z],
     \end{displaymath}
     to illustrate that the initial reduction of the standard basis
     ist important for determining the inequalities and
     equations of the corresponding Gr\"obner cone (see
     Algorithm~\ref{alg:groebnerConeNoWeight}).

     Note that the generating set is a standard basis w.r.t.~$>$, but it is not yet
     initially reduced as the terms $t^2y$ in $g_1$ and $tx$ in $g_2$
     still lie in $\lt_>(I)=\langle
     2,x,y\rangle$. Consequently, these two terms yield meddling
     inequalities, so that \lang{(the overline denoting the closure in the
     Euclidean topology)}
     \begin{displaymath}
       C:=\overline{\{w\in \RR_{< 0}\times\RR^3\mid \initial_w(g_i)=\initial_v(g_i) \text{ for } i=0,1,2 \}}\subsetneq C_v(I).
     \end{displaymath}
     Ignoring $g_0$, as it yields no non-trivial inequalities in
     $\RR_{\leq 0}\times\RR^3$, $C$ is
     the polyhedral cone given by the inequalities (see
     Figure~\ref{fig:groebnerConeNotInitiallyReducedInequalities1})
     \begin{displaymath}
       \initial_w(g_1) = x \;\;\; \Longleftrightarrow \;\;\; \begin{cases} \deg_w(x)\geq \deg_w(t^2y) \\ \deg_w(x)\geq \deg_w(t^3z) \end{cases} \Longleftrightarrow \;\;\; \begin{cases} w_1\geq 2w_0+w_2 \\ w_1\geq 3w_0+w_3 \end{cases}
     \end{displaymath}
     \lang{and}
     \begin{displaymath}
       \initial_w(g_2) = y \;\;\; \Longleftrightarrow \;\;\; \begin{cases} \deg_w(y)\geq \deg_w(tx) \\ \deg_w(y)\geq \deg_w(t^2z) \end{cases} \Longleftrightarrow \;\;\; \begin{cases} w_2\geq w_0+w_1 \\ w_2\geq 2w_0+w_3 \end{cases}
     \end{displaymath}
     \begin{figure}[h]
       \centering
       \begin{tikzpicture}[scale=0.7, every node/.style={transform shape}]
         \fill[color=blue!20] (-3,-3) rectangle (6.5,5);
         \draw[->] (-2,0) -- (4,0) node [anchor=west] {$w_1$};
         \draw[->] (0,-2) -- (0,4) node [anchor=west] {$w_2$};
         \fill (0,0) circle (2pt);
         \node[anchor=north west,font=\scriptsize] at (0,0) {$(-1,0,0,1)$};

         \draw[red] (-2,-1) -- (3,4) node [black, font=\scriptsize, anchor=south] {$\{w_1>-2+w_2\}$};
         \draw[thick,red,->] (2,3) -- (2.3,2.7);
         \draw[red] (-1,-2) -- (-1,4) node [black, font=\scriptsize, anchor=south] {$\{w_1>-3+1\}$};
         \draw[thick,red,->] (-1,-1) -- (-0.5,-1);
         \fill (-1,0) circle (1pt);
         \node[anchor=south east,font=\scriptsize] at (-1,0) {$-2$};

         \fill (1,0) circle (2pt);
         \node[anchor=south west,font=\scriptsize] at (1,0) {$(-1,2,0,1)$};

         \node[anchor=south east,font=\scriptsize] at (5,-3) {$\RR_{\leq 0}\times\RR^3\cap\{w_0=-1,w_3=1\}$};
       \end{tikzpicture}
       \begin{tikzpicture}[scale=0.7, every node/.style={transform shape}]
         \fill[color=blue!20] (-3,-3) rectangle (6.5,5);
         \draw[->] (-2,0) -- (4,0) node [anchor=west] {$w_1$};
         \draw[->] (0,-2) -- (0,4) node [anchor=west] {$w_2$};
         \fill (0,0) circle (2pt);
         \node[anchor=south east,font=\scriptsize] at (0,0) {$(-1,0,0,1)$};
         \draw[red] (-1.5,-2) -- (4,3.5) node [black, font=\scriptsize, anchor=west] {$\{w_2>-1+w_1\}$};
         \draw[thick,red,->] (3,2.5) -- (2.7,2.8);

         \draw[red] (-2,-0.5) -- (4,-0.5) node [black, font=\scriptsize, anchor=west] {$\{w_2>-2+1\}$};
         \draw[thick,red,->] (-1,-0.5) -- (-1,0);
         \fill (0,-0.5) circle (1pt);
         \node[anchor=north west, font=\scriptsize, yshift=0.05cm] at (0,-0.5) {$-1$};

         \fill (1,0) circle (2pt);
         \node[anchor=north west,font=\scriptsize] at (1,0) {$(-1,2,0,1)$};

         \node[anchor=south east,font=\scriptsize] at (6.5,-3) {$\RR_{\leq 0}\times\RR^3\cap\{w_0=-1,w_3=1\}$};
       \end{tikzpicture}

       \caption{inequalities given by $\initial_w(g_1)=x$ resp.~$\initial_w(g_2)=y$}
       \label{fig:groebnerConeNotInitiallyReducedInequalities1}
     \end{figure}

     Clearly, $w:=(-1,2,0,1)\not\in C$, even though
     $\initial_w(I) = \initial_v(I)$, since
     \begin{align*}
       \initial_w(g_1-t^2\cdot g_2) &= \initial_w(x-t^3x+t^3z-t^4z)= x,\\
       \initial_w(g_2-t\cdot g_1)   &=\initial_w(y-t^3y+t^2z-t^4z)=y,
     \end{align*}
     implying that $w\in C_v(I)$. Replacing $\{g_1,g_2\}$ with the
     initially reduced standard basis $\{g_1-t^2\cdot g_2,g_2-t\cdot
     g_1\}$, we see that we are replacing the unnecessary inequalities
     above, induced by $t^2y$ and $tx$, with the redundant inequalities
     of Example~\ref{ex:31}, induced by $t^3x$, $t^3y$ and $t^4z$.
   \end{example}


   \section{Initially reduced standard bases}\label{sec:initiallyreduced}

   In this section, we present an algorithm for the initial reduction of
   a polynomial ${\ux}$-homogeneous standard basis  in finite time. For the
   sake of simplicity, we will restrict ourselves to a special case
   which is of particular interest for the computation of tropical
   varieties over the $p$-adic numbers (see~\cite{MR15b}), though the basic ideas
   behind the  algorithm can be generalised.

   Throughout this section we assume that $K$ is some field\label{page:initialReduction}
   with non-trivial discrete valuation, $\mathfrak K$ its residue
   field, $R_\nu$ its discrete valuation ring, $p\in R_\nu$ a
   uniformising parameter and $R\subset R_\nu$ a dense noetherian
   subring with $p\in R$. Both $K$ and $R_\nu$ are assumed to be
   complete, so that we have exact sequences
   \begin{center}
     \begin{tikzpicture}[descr/.style={fill=white,inner sep=2pt}]
       \matrix (m) [matrix of math nodes, row sep=2em,
       column sep=2.5em, text height=1.5ex, text depth=0.25ex]
       { 0 & \langle p-t \rangle\cdot \Rt_{\langle p-t \rangle} [{\ux}]
         & \Rt_{\langle p-t \rangle} [{\ux}]
         & K[{\ux}] & 0, \\
         0 & \langle p-t \rangle\cdot \Rtx & \Rtx & R_\nu[{\ux}] & 0. \\ };
       \path[->,font=\scriptsize]
       (m-1-1) edge (m-1-2)
       (m-1-2) edge (m-1-3)
       (m-1-3) edge (m-1-4)
       (m-1-4) edge (m-1-5)
       (m-2-1) edge (m-2-2)
       (m-2-2) edge (m-2-3)
       (m-2-3) edge node[below] {$t\longmapsto p$} (m-2-4)
       (m-2-4) edge (m-2-5)
       (m-2-2) edge (m-1-2)
       (m-2-3) edge (m-1-3)
       (m-2-4) edge (m-1-4);
     \end{tikzpicture}
   \end{center}
   and $R/\langle p \rangle = \mathfrak{K}$.\label{page:spezialannahme}
   Moreover,  we still require
   that linear equations in $R$ are solvable, so
   that standard bases in $\Rtx$ exist and are computable. If $R=\ZZ$
   is the ring of integers, $p\in\ZZ$ a prime number and $K=\QQ_p$ the
   field of $p$-adic numbers, all properties are fulfilled
   (see~\cite{MR15b} for further interesting examples).
   We then fix the preimage $I\unlhd\Rtx$ of some homogeneous ideal in $K[{\ux}]$, which
   in particular implies that $I$ is ${\ux}$-homogeneous and $p-t\in
   I$. It is our aim to provide an algorithm which computes an
   initially reduced standard basis of $I$ w.r.t.~some $t$-local
   monomial ordering $>$ on $\Mon(t,{\ux})$, provided that the ideal
   $I$ is generated by polynomials. See Example~\ref{ex:spezialannahmen} for an
   example.

   This section has a simple monolithic structure. Because our ideals are
   all ${\ux}$-homo\-ge\-neous, the problems that commonly arise when lacking a
   well-ordering actually root in the inhomogeneity in $t$ alone. It
   turns out that these problems can be circumvented by reducing
   w.r.t.~$p-t$ diligently. Hence we begin with an algorithm
   dedicated to  that.
   Next, we continue with an algorithm for reducing a set of elements
   of the same ${\ux}$-degree w.r.t.~themselves and $p-t$. Having all
   elements sharing the same ${\ux}$-degree makes the inhomogeneity in $t$
   easy to handle. Using it, we construct an algorithm for reducing a set
   of elements of the same ${\ux}$-degree w.r.t.~themselves, $p-t$ and
   another set of elements of strictly lower ${\ux}$-degree. This is the part
   in which the difficulty of our lack of well-ordering becomes
   apparent. We then conclude the section with
   Algorithm~\ref{alg:inred} for computing an initially reduced
   standard basis by reducing a standard basis w.r.t.~itself.

   \lang{We will now formulate an algorithm to initially reduce w.r.t.~$p-t$.}

   \begin{algorithm}[$(p-t)$-Reduce]\label{alg:pReduce} \ \vspace*{-3ex}
     \begin{algorithmic}[1]
       \REQUIRE{$(g,>)$, where $>$ is a $t$-local monomial ordering and $g\in \Rtxp$ ${\ux}$-homo\-ge\-neous.}
       \ENSURE{$g'\in\Rtxp$ ${\ux}$-homogeneous with
         $\langle p-t,g'\rangle=\langle p-t,g\rangle\unlhd\Rtxp$, $\lt_>(g')=\lt_>(g)$ and initially reduced w.r.t.~$p-t$ under $>$\lang{, i.e. no term of $\tail_>(g')$ is divisible by $\lt_>(p-t)=p$}.}
       \STATE Suppose $g=\sum_{\alpha} g_\alpha\cdot {\ux}^\alpha$ with $g_\alpha\in R[t]$ and $\lt_>(g)=\lt_>(g_{\gamma})\cdot {\ux}^{\gamma}$.
       \STATE Set $g':=g_{\gamma}\cdot {\ux}^{\gamma}$ and $g'':=g-g_{\gamma}\cdot {\ux}^{\gamma}$, so that $g=g'+g''$.
       \WHILE{$g''\neq 0$}
       \STATE Suppose $g''=\sum_{\alpha} g''_\alpha\cdot {\ux}^\alpha$ with $g''_\alpha\in R[t]$ and $\lt_>(g'')=\lt_>(g''_{\gamma})\cdot {\ux}^{\gamma}$.
       \IF{$p\mid \lt_>(g''_\gamma)$}
       \STATE Let $l:=\max\{m\in\NN\mid p^m \text{ divides } \lt_>(g''_\gamma)\}>0$.
       \STATE Set $g'':=g''-\frac{\lt_>(g''_\gamma)}{p^l}\cdot (p^l-t^l)$.
       \ELSE
       \STATE Set $g':=g'+g''_\gamma\cdot {\ux}^\gamma$ and $g'':=g''-g''_\gamma\cdot {\ux}^\gamma$.
       \ENDIF
       \ENDWHILE
       \RETURN{$g'$}
     \end{algorithmic}
   \end{algorithm}
   \begin{proof}
     \emph{Termination:} We need to show that $g''=0$ eventually. Since
     all changes to $g''$ during a single iteration of the \texttt{while}
     loop happen at a distinct monomial in ${\ux}$, namely that of
     $\lm_>(g'')$, we may assume for our argument that all terms of $g''$
     have the same monomial in ${\ux}$. Suppose, in the beginning of an
     iteration,
     \begin{displaymath}
       g''=(c_{i_1}t^{i_1}+\ldots+c_{i_j}\cdot t^{i_j})\cdot {\ux}^\gamma \text{ with } i_1<\ldots<i_j.
     \end{displaymath}
     Now if $p\nmid \lt_>(c_{i_1})$, then $g''$ will be set to $0$ in
     Step $9$ and the algorithm terminates. If $p\mid \lt_>(c_{i_1})$, we
     substitute the term $c_{i_1}\cdot t^{i_1}{\ux}^\gamma$ by the term
     $c_{i_1}/p^l \cdot t^{i_1+l}{\ux}^\gamma$ in Step $7$, increasing the
     minimal $t$-degree strictly.

     Let $\nu_p(c):=\max\{m\in\NN\mid p^m \text{ divides } c\}$ denote
     the $p$-adic valuation on $R$, so that $l=\nu_p(c_{i_1})$, and
     consider the valued degree of $g''$ defined by
     \begin{displaymath}
       \max\{ \nu_p(c_{i_1})+\deg(t^{i_1}),\ldots, \nu_p(c_{i_j})+\deg(t^{i_j}) \}.
     \end{displaymath}
     This is a natural upper bound on the $t$-degree of our substitute, and hence also for the $t$-degree of all terms in our new $g''$.

     If the monomial of the substitute, $t^{i_1+l}{\ux}^\gamma$, does not
     occur in the original $g''$, then this upper bound remains the same
     for our new $g''$. If it does occur in the original $g''$, then this
     valued degree might increase depending on the sum of the
     coefficients, however the number of terms in $g''$ strictly
     decreases.

     Because $g''$ has only finitely many terms to begin with, this upper
     bound may therefore only increase a finite number of times. And
     since the minimal $t$-degree is strictly increasing, if $g''$ is not
     set to $0$, our algorithm terminates eventually.

     \emph{Correctness:} It is clear that $g'$ remains polynomial and
     ${\ux}$-homogeneous. And the only term of $g'$ that might be divisible
     by $\lt_>(p-t)=p$ is $\lt_>(g')=\lt_>(g)$, since all other terms
     passed the check in Step $5$ negatively. Hence $g'$ is initially
     reduced w.r.t.~$p-t$ under $>$.
   \end{proof}

   With this, we can begin formulating an algorithm for initially reducing a set of elements which are ${\ux}$-homogeneous of same degree in ${\ux}$.

   \begin{algorithm}[initial reduction, same degree in ${\ux}$]\label{alg:inred00} \ \vspace*{-3ex}
     \begin{algorithmic}[1]
       \REQUIRE{$(G,>)$, where $>$ is a $t$-local monomial ordering and
         $G=\{g_1,\ldots,g_k\} \subseteq \Rtxp$ a finite subset such that
         \begin{enumerate}[leftmargin=*, itemsep=0pt]
         \item $g_1,\ldots,g_k$ ${\ux}$-homogeneous of the same ${\ux}$-degree,
         \item $\lc_>(g_i)=1$ for $i=1,\ldots,k$,
         \item $\lm_>(g_i)\neq\lm_>(g_j)$ for $i\neq j$.
         \end{enumerate}}
       \ENSURE{$G'=\{g_1',\ldots,g_k'\}\subseteq \Rtxp$ such that
         \begin{enumerate}[leftmargin=*, itemsep=0pt]
         \item $g_1',\ldots,g_k'$ ${\ux}$-homogeneous of the same ${\ux}$-degree,
         \item $\lt_>(g_i')=\lt_>(g_i)$ for $i=1,\ldots,k$,
         \item $G'$ initially reduced w.r.t.~itself and $p-t$,
         \item $\langle p-t,g_1,\ldots,g_k\rangle=\langle p-t,g_1',\ldots,g_k'\rangle\unlhd\Rtx$.
         \end{enumerate}}
       \FOR{$i=1,\ldots,k$}
       \STATE Run $g_i:=(p-t)\text{-Reduce}(g_i,>)$.
       \ENDFOR
       \STATE Reorder $G=\{g_1,\ldots,g_k\}$ such that
       $\lm_>(g_1)>\ldots>\lm_>(g_k)$,
       and suppose
       \begin{displaymath}
         g_i:=\sum_{\alpha\in\NN} g_{i,\alpha}\cdot {\ux}^\alpha \text{ with } g_{i,\alpha}\in\Rt \text{ and } \lt_>(g_i)=t^{\beta_i} {\ux}^{\alpha_i}.
       \end{displaymath}
       \FOR{$i=1,\ldots,k-1$}
       \FOR{$j=i+1,\ldots,k$}
       \IF{$g_{j,\alpha_i}\neq 0$} 
       \STATE Set
       \begin{displaymath}
         g_j := \frac{g_{i,\alpha_i}}{t^{\beta_i}}\cdot g_j
         - \frac{g_{j,\alpha_i}}{t^{\beta_i}}\cdot g_i.
       \end{displaymath}
       \STATE Run $g_j:=(p-t)\text{-Reduce}(g_j,>)$.
       \ENDIF
       \ENDFOR
       \ENDFOR
       \FOR{$i=1,\ldots,k-1$}
       \FOR{$j=i+1,\ldots,k$}
       \IF{$t^{\beta_j} \mid g_{i,\alpha_j}$}
       \STATE Set
       \begin{displaymath}
         g_i := \frac{g_{j,\alpha_j}}{t^{\beta_j}}\cdot g_i
         - \frac{g_{i,\alpha_j}}{t^{\beta_j}}\cdot g_j.
       \end{displaymath}
       \STATE Run $g_i:=(p-t)\text{-Reduce}(g_i,>)$.
       \ENDIF
       \ENDFOR
       \ENDFOR
       \RETURN{$G'=\{g_1,\ldots,g_k\}$.}
     \end{algorithmic}
   \end{algorithm}
   \begin{proof}
     For the correctness of the instructions note that, by definition and
     because $>$ is $t$-local, $g_{j,\alpha_j}$ is divisible by
     $t^{\beta_j}$ and $g_{i,\alpha_i}$ is divisible by $t^{\beta_i}$ in
     Step $7$. From the assumption in Step $11$ it follows that
     $g_{i,\alpha_j}$ in Step $12$ will be divisible by
     $t^{\beta_j}$. Observe that due to
     the reordering in Step $3$ and $\lm_>(g_{j,\alpha_i})\cdot
     {\ux}^{\alpha_i}$ being a monomial
     in $g_j$ we have for $i<j$:
     \begin{displaymath}
       t^{\beta_i}\cdot {\ux}^{\alpha_i} = \lm_>(g_i) > \lm_>(g_j) >
       \lm_>(g_{j,\alpha_i})\cdot  {\ux}^{\alpha_i}.
     \end{displaymath}
     Now since $>$ is $t$-local,
     $t^{\beta_i}$ divides $\lm_>(g_{j,\alpha_i})$, hence also $g_{j,\alpha_i}$.

     It is clear that the algorithm terminates since it only consists of
     a finite number of steps, and, for the correctness, that the output
     is ${\ux}$-homogeneous, polynomial and generates the same ideal
     as the input.

     Next, we show that the leading terms of the $g_i$ are
     preserved. Observe that in Step 7 we have
     $\lm_>(\frac{g_{i,\alpha_i}}{t^{\beta_i}})=1$ by definition
     and $\lm_>(\frac{g_{j,\alpha_j}}{t^{\beta_i}})< 1$ by the previous argument.
     Due to the assumption that $\lc_>(g_i)=\lc_>(g_{i,\alpha_i})=1$ we therefore have
     \begin{displaymath}
       \lt_>(g_j)=\lt_>\left(\frac{g_{i,\alpha_i}}{t^{\beta_i}}\cdot g_j\right)
     \end{displaymath}
     and
     \begin{displaymath}
       \lm_>(g_j)>\lm_>(g_{j,\alpha_i})\cdot {\ux}^{\alpha_i}=\lm_>\left(\frac{g_{j,\alpha_i}}{t^{\beta_i}}\cdot g_i\right).
     \end{displaymath}
     In Step 12 we similarly have $\lm_>(\frac{g_{j,\alpha_j}}{t^{\beta_j}})=1$
     and $\lm_>(\frac{g_{i,\alpha_j}}{t^{\beta_j}})\leq 1$, thus
     \begin{displaymath}
       \lt_>(g_i)=\lt_>\left(\frac{g_{j,\alpha_j}}{t^{\beta_j}}\cdot g_i\right)
     \end{displaymath}
     and
     \begin{displaymath}
       \lm_>(g_i)>\lm_>(g_{i,\alpha_j})\cdot {\ux}^{\alpha_j}=\lm_>\left(\frac{g_{i,\alpha_j}}{t^{\beta_j}}\cdot g_j\right).
     \end{displaymath}
     On the whole, the leading terms of the $g_1,\ldots,g_k$ remain unchanged.

     The output is initially reduced w.r.t.~$p-t$. For
     that note that $p$ does neither divide the leading terms as they are monic
     nor the latter terms because every element of the
     output was sent through the Algorithm~\ref{alg:pReduce}.

     To see that the output $G'$ is initially reduced w.r.t.~itself, observe that the first pair of nested \texttt{for}
     loops eliminates all terms in $g_j$ with ${\ux}^{\alpha_i}$ for $i<j$. In particular, each $g_j$ is
     initially reduced w.r.t.~$g_1,\ldots,g_{j-1}$.

     Additionally, it will stay reduced w.r.t.~$g_1,\ldots,g_{j-1}$ in the second pair of nested \texttt{for}
     loops, because $g_{j+1},\ldots,g_{k}$ contain no monomial
     ${\ux}^{\alpha_i}$, $i<j$, either.

     Moreover, once $g_i$ is initially  reduced w.r.t.~$g_j$ for
     $i<j$ in Step $12$, reducing it initially w.r.t.~say
     $g_{j+1}$ will not change that out of two reasons. First, $g_{j+1}$
     contains no term with ${\ux}^{\alpha_j}$, hence adding a multiple of it
     to $g_i$ is unproblematic. Secondly,
     $\lt_>(g_{j,\alpha_j}/t^{\beta_j})=1$, which means multiplying $g_i$
     by it will not change $\lt_>(g_{i,\alpha_j})$. So if $t^{\beta_j}$
     does not divide $g_{i,\alpha_j}$ before, because $g_i$ is initially
     reduced w.r.t.~$g_j$, it does not divide $g_{i,\alpha_j}$
     after as well.

     This shows that the constant changes to $g_i$ in the second pair of
     nested \texttt{for} loops are unproblematic. Once $g_i$ has been
     initially reduced w.r.t.~$g_j$, it will stay that way while
     being reduced initially w.r.t.~$g_{j+1},\ldots,g_k$.
   \end{proof}

   \begin{example}
     Let $p=2$ and consider the set $G=\{g_1,g_2,g_3\}\subseteq \ZZ\llbracket t\rrbracket [x_1,x_2,x_3]$ with
     \begin{align*}
       g_1&:= x_1^2+t x_2^2-t^2x_3^2, \\
       g_2&:= x_2^2+tx_1^2+tx_3^2+t^2x_3^2=x_2^2+tx_1^2+(t+t^2)x_3^2, \\
       g_3&:= t^3x_3^2+t^4x_1^2+t^4x_2^2+t^5x_2^2 = t^3x_3^2+t^4x_1^2+(t^4+t^5)x_2^2,
     \end{align*}
     and the weighted ordering $>=>_w$ on $\Mon(t,{\ux})$ with weight vector
     $(-1,1,1,1)\in\RR_{<0}\times\RR^3$ and the $t$-local lexicographical
     ordering with $x_1>x_2>x_3>1>t$ as tiebreaker. \lang{The $g_i$, $i=1,2,3$, as
     well as their terms have already been ordered above.}

     We can illustrate the process with the aid of the following
     $3\times 3$-matrix:
     \begin{displaymath}
       \begin{pmatrix}
         1 & t & -t^2 \\ t & 1 & t+t^2 \\ t^4 & t^4+t^5 & t^3
       \end{pmatrix}.
     \end{displaymath}
     The entry in position $(i,j)$ contains the $R\llbracket t\rrbracket$-coefficient of
     $g_i$ w.r.t.~the $\ux$-monomial in the leading term of
     $g_j$.

     In the first pass, we begin by taking $g_1$ and reducing $g_2$ and
     $g_3$ w.r.t.~it. To eliminate the term $tx_1^2$ in $g_2$
     and $t^4x_1^2$ in $g_3$ we set
     \begin{align*}
       g_2:= & \,g_2 - t\cdot g_1 = (x_2^2+tx_1^2+tx_3^2+t^2x_3^2)-t\cdot (x_1^2+t x_2^2-t^2x_3^2) \\
       = & \,(1-t^2)\cdot x_2^2+(t+t^2+t^3)\cdot x_3^2, \\
       g_3:= & \,g_3 - t^4\cdot g_1 = (t^3x_3^2+t^4x_1^2+(t^4+t^5)x_2^2) - t^4\cdot (x_1^2+t x_2^2-t^2x_3^2) \\
       = & \,(t^3+t^6)\cdot x_3^2+t^4\cdot x_2^2.
     \end{align*}
     Note that both $g_2$ and $g_3$ remain initially reduced w.r.t.~$2-t$.
     \begin{center}
       \begin{tikzpicture}
         \matrix (m) [matrix of math nodes, row sep=1em, column sep=5em, text height=1.5ex, text depth=0.25ex, column 3/.style={anchor=base west}]
         { g_1 & \textcolor{red}{g_2} & \textcolor{red}{g_3} \\ };
         \path[->,font=\scriptsize]
         (m-1-1) edge [bend right=15] (m-1-2)
         (m-1-1) edge [bend left=15] (m-1-3);
       \end{tikzpicture}
     \end{center}
     \begin{displaymath}
       \begin{pmatrix} 1 & t & -t^2 \\ \textcolor{red}0 & 1-t^2 & t+t^2+t^3 \\ \textcolor{red}0 & t^4 & t^3+t^6 \end{pmatrix}
     \end{displaymath}
     Next, we take $g_2$ and reduce $g_3$ w.r.t.~it, i.e.
     \begin{align*}
       g_3:=&\, (1-t^2)\cdot g_3 - t^4\cdot g_2 \\
       =&\, (1-t^2)\cdot((t^3+t^6)x_3^2+t^4x_2^2) -t^4\cdot((1-t^2) x_2^2+(t+t^2+t^3)x_3^2)\\
       =&\, (t^3-2t^5-t^7-t^8)\cdot x_3^2.
     \end{align*}
     And even though $g_3$ contains a term divisible by $2$, it still remains initially reduced w.r.t.~$2-t$.
     \begin{center}
       \begin{tikzpicture}
         \matrix (m) [matrix of math nodes, row sep=1em, column sep=5em, text height=1.5ex, text depth=0.25ex, column 3/.style={anchor=base west}]
         { g_1 & g_2 & \textcolor{red}{g_3} \\ };
         \path[->,font=\scriptsize]
         (m-1-2) edge [bend left=15] (m-1-3);
       \end{tikzpicture}
     \end{center}
     \begin{displaymath}
       \begin{pmatrix} 1 & t & -t^2 \\ 0 & 1-t^2 & t+t^2+t^3 \\ 0 & \textcolor{red}{0} & t^3-2t^5-t^7-t^8 \end{pmatrix}
     \end{displaymath}

     This concludes our first pass. For the second pass, we begin by
     taking $g_1$ and reducing it w.r.t.~first $g_2$ and then
     $g_3$. Reducing $g_1$ w.r.t.~$g_2$ yields
     \begin{align*}
       g_1:=&\,(1-t^2) \cdot g_1-t\cdot g_2 \\
       =&\, (1-t^2)\cdot (x_1^2+t x_2^2-t^2x_3^2) - t\cdot ((1-t^2) x_2^2+(t+t^2+t^3)x_3^2) \\
       =&\,(1-t^2) \cdot x_1^2 +(-2t^2-t^3)\cdot x_3^2
     \end{align*}
     and reducing that w.r.t.~$2-t$ we obtain
     \begin{displaymath}
       g_1:=g_1 -(-t^2-t^3)x_3^2\cdot (2-t) = (1-t^2) \cdot x_1^2 -t^4x_3^2.
     \end{displaymath}
     Reducing $g_1$ w.r.t.~$g_3$ yields,
     \begin{align*}
       g_1:=&\,(1-2t^2-t^4-t^5)\cdot g_1 - t\cdot g_3 = (1-2t^2-t^4-t^5)(1-t^2) x_1^2\\
       =&\,(1-3t^2+t^4-t^5+t^6+t^7) \cdot x_1^2,
     \end{align*}
     which is initially reduced w.r.t.~$2-t$.
     \begin{center}
       \begin{tikzpicture}
         \matrix (m) [matrix of math nodes, row sep=1em, column sep=5em, text height=1.5ex, text depth=0.25ex, column 3/.style={anchor=base west}]
         { \textcolor{red}{g_1} & g_2 & g_3 \\ };
         \path[->,font=\scriptsize]
         (m-1-2) edge [bend left=15] (m-1-1)
         (m-1-3) edge [bend right=15] (m-1-1);
       \end{tikzpicture}
     \end{center}
     \begin{displaymath}
       \begin{pmatrix}
         1-3t^2+t^4-t^5+t^6+t^7 & \textcolor{red}0 & \textcolor{red}0 \\ 0 & 1-t^2 & t+t^2+t^3 \\ 0 & 0 & t^3-t^6-t^7-t^8
       \end{pmatrix}
     \end{displaymath}
     Finally, note that while $g_2$ has a term $t^3 x_3^2$ divisible by
     the leading term $t^3x_3$ of $g_3$, it is still initially reduced
     w.r.t.~$g_3$. This concludes our second pass and we obtain
     the initially reduced set
     \begin{align*}
       g_1&= (1-5t^2+3t^4-t^5+t^6+t^7) \cdot x_1^2, \\
       g_2&= (1-t^2)\cdot x_2^2+(t+t^2+t^3)\cdot x_3^2, \\
       g_3&= (t^3-2t^5-t^7-t^8)\cdot x_3^2.
     \end{align*}

     Observe that it is possible to reduce the number of terms at the
     cost of the coefficient size, by substituting $p$ for some of the
     $t$. One alternative initially reduced set with the same leading
     monomials as above would therefore be
     \begin{displaymath}
       g_1:= 165 \cdot x_1^2,\quad g_2:= -3\cdot x_2^2+7t\cdot x_3^2 \quad \text{and} \quad g_3:= -55t^3\cdot x_3^2.
     \end{displaymath}
   \end{example}

   Next, we need to discuss how to reduce a set $H$ of ${\ux}$-homogeneous
   elements of the same degree in ${\ux}$ w.r.t.~themselves and a set
   $G$ of ${\ux}$-homogeneous elements of lower degree. The simplest way
   is multiplying the elements of $G$ up to the same degree in ${\ux}$
   as the elements of $H$ in all possible combinations and using
   Algorithm~\ref{alg:inred00} on the resulting set. This resembles a
   brute force method in which we directly summon the worst case scenario
   to be resolved. \lang{Consequently, it is an algorithm which is good in
   cases in which the worst case is unavoidable.

   \begin{algorithm}[initial reduction, all at once]\label{alg:inred02}\ \vspace*{-3ex}
     \begin{algorithmic}[1]
       \REQUIRE{$(G,H,>)$, where $>$ a $t$-local monomial ordering,
         $H=\{h_1,\ldots,h_k\}$ and  $G$ finite subsets of $\Rtxp$ such that
         \begin{enumerate}[leftmargin=*]
         \item $h_1,\ldots,h_k$ are ${\ux}$-homogeneous of the same ${\ux}$-degree $d$,
         \item all $g\in G$ are ${\ux}$-homogeneous of ${\ux}$-degree less than $d$,
         \item $\lc_>(h_i)=1$, for $i=1,\ldots,k$, and $\lc_>(g)=1$ for all $g\in G$,
         \item $\lm_>(h_i)\neq \lm_>(h_j)$ for $i\neq j$,
         \item $\lm_>(h_i)\notin\langle \lm_>(g)\mid g\in G\rangle$ for $i=1,\ldots,k$.
         \end{enumerate}}
       \ENSURE{$H'=\{h_1',\ldots,h_k'\}\subseteq\Rtxp$ such that
         \begin{enumerate}[leftmargin=*]
         \item $h_1',\ldots,h_k'$ are ${\ux}$-homogeneous of the same ${\ux}$-degree $d$,
         \item $\lt_>(h_i')=\lt_>(h_i)$ for $i=1,\ldots,k$,
         \item $H'$ initially reduced w.r.t. $G$ and itself,
         \item $\langle p-t,G,H\rangle=\langle p-t,G,H'\rangle\unlhd\Rtx$.
         \end{enumerate}}
       \STATE Set $E:=\emptyset$.
       \FOR{$\alpha\in\NN^n, |\alpha|=d$}
       \IF{$t^\beta {\ux}^\alpha\in\lt_>(G)$ for some $\beta\in\NN$
         and $t^\beta{\ux}^\alpha\not\in\lt(E)$}
       \STATE Pick $g\in G$ with $\lt_>(g)\mid t^\beta {\ux}^\alpha$ for some minimal $\beta\in\NN$.
       \STATE Set $E:=E\cup\left\{\frac{t^\beta {\ux}^\alpha}{\lt_>(g)}\cdot g\right\}$.
       \ENDIF
       \ENDFOR
       \STATE Reduce $H\cup E$ initially with Algorithm~\ref{alg:inred00}.
       \RETURN{$H$}
     \end{algorithmic}
   \end{algorithm}
   \begin{proof}
     Due to the necessary conditions of this algorithm, $H\cup E$
     satisfies the necessary conditions for Algorithm~\ref{alg:inred00}. The correctness of this algorithm now follows
     from the correctness of Algorithm~\ref{alg:inred00}.
   \end{proof}

}
   A more sophisticated method multiplies the elements of $G$ up
   to the same degree in ${\ux}$ as the elements of $H$ when they are
   needed. In the optimal case, we can reduce the complexity drastically
   with this strategy, in the worst case we are only delaying the
   inevitable.

   \begin{algorithm}[initial reduction, step by step]\label{alg:inred01} \ \vspace*{-3ex}
     \begin{algorithmic}[1]
       \REQUIRE{$(G,H,>)$, where $>$ a $t$-local monomial ordering,
         $H=\{h_1,\ldots,h_k\}$ and $G$ finite subsets of $\Rtxp$ such that
         \begin{enumerate}[leftmargin=*, itemsep=0pt]
         \item $h_1,\ldots,h_k$ are ${\ux}$-homogeneous of the same ${\ux}$-degree $d$,
         \item all $g\in G$ are ${\ux}$-homogeneous of ${\ux}$-degree less than $d$,
         \item $\lc_>(h_i)=1$, for $i=1,\ldots,k$, and $\lc_>(g)=1$ for all $g\in G$,
         \item $\lm_>(h_i)\neq \lm_>(h_j)$ for $i\neq j$,
         \item $\lm_>(h_i)\notin\langle \lm_>(g)\mid g\in G\rangle$ for $i=1,\ldots,k$.
         \end{enumerate}}
       \ENSURE{$H'=\{h_1',\ldots,h_k'\}\subseteq\Rtxp$ such that
         \begin{enumerate}[leftmargin=*, itemsep=0pt]
         \item $h_1',\ldots,h_k'$ are ${\ux}$-homogeneous of the same ${\ux}$-degree $d$,
         \item $\lt_>(h_i')=\lt_>(h_i)$ for $i=1,\ldots,k$,
         \item $H'$ initially reduced w.r.t. $G$ and itself,
         \item $\langle p-t,G,H\rangle=\langle p-t,G,H'\rangle\unlhd\Rtx$.
         \end{enumerate}}
       \STATE Reduce $H$ initially using Algorithm~\ref{alg:inred00} and set $E=\emptyset$.
       \STATE Suppose $h_i=\sum_{\alpha\in\NN^n}h_{i,\alpha}\cdot {\ux}^\alpha$ with $h_{i,\alpha}\in\Rt$,
       create the disjoint union
       \begin{displaymath}
         T:=\{(\lt_>(h_{i,\alpha})\cdot  {\ux}^\alpha,i)\mid
         \alpha\in\NN^n \text{ and } \lt_>(h_{i,\alpha})\cdot  {\ux}^\alpha<\lt_>(h_i)\},
       \end{displaymath}
       a working list of terms to be checked for potential reduction w.r.t.~$G$.
       \WHILE{$T\neq\emptyset$}
       \STATE Pick $(s,i)\in T$ with $\lm_>(s)$ maximal.
       \IF{$\lt_>(g)\mid s$ for some $g\in G$}
       \STATE Pick $g\in G$, $\lt_>(g)\mid s$, and set $E:=E\cup\left\{\frac{\lm_>(s)}{\lm_>(g)}\cdot g\right\}$.
       \STATE Reduce $H\cup E$ initially using Algorithm~\ref{alg:inred00}.
       \STATE Update the working list:
       \begin{displaymath}
         T:=\{(\lt_>(h_{i,\alpha})\cdot  {\ux}^\alpha, i)\mid
         \alpha\in\NN^n \text{ and } \lm_>(h_{i,\alpha})\cdot  {\ux}^\alpha<
         \lm_>(s) \}.
       \end{displaymath}
       \ELSE
       \STATE Set $T:=T\setminus\{(h_i,s)\}$.
       \ENDIF
       \ENDWHILE
       \RETURN{$H$}
     \end{algorithmic}
   \end{algorithm}
   \begin{proof}
     For the termination note that in each iteration of the while loop
     either the set of extra polynomials $E$ increases or the working
     list $T$ decreases.
     Also because each $s$ is chosen to be maximal, each other term in the
     working list $T$ with the same ${\ux}$-monomial must have a higher
     $t$-degree and is therefore eliminated alongside $s$ in the initial
     reduction of $H\cup E$. Because the updated $T$ only includes relevant
     terms smaller than $s$, the ${\ux}$-monomial of $s$ is effectively
     eliminated in all working lists to follow. Hence each elements of $E$
     will always have a distinct ${\ux}$-monomial which is of degree $d$. Thus
     $E$ has a maximal size after which the algorithm will terminate in a
     finite number of steps.

     For the correctness of the instructions, observe that $H\cup E$
     satisfies the conditions for Algorithm~\ref{alg:inred00} by
     assumption. For the correctness of the output, it is obvious that
     the leading terms of $H$ are preserved, that $H$ is initially
     reduced w.r.t.~itself and that its elements are
     ${\ux}$-homogeneous as well as polynomial. To show that $H$ is initially
     reduced w.r.t.~$G$, observe that, apart from the terms
     eliminated, any term altered in the initial reduction of $H\cup E$
     is strictly smaller than $s$. Because $s$ was chosen to be maximal,
     the updated working list therefore contains all relevant terms that
     have been altered or that have yet to be checked for reduction. Thus
     in the output any relevant term has been negatively checked for
     divisibility by an element of $G$.
   \end{proof}

   \begin{remark}
     Note that in Step 6 of Algorithm~\ref{alg:inred01}, we
     multiply $g$ by a power of $t$ even though it is not necessary
     for correctness. The reason is as follows:

     Recall Algorithm~\ref{alg:inred00}, which consists of two big nested
     \texttt{for} loops. In the first pass from Step 4 to 8 we take each
     $g_i$, $i=1,\ldots,k-1$, and reduce all $g_j$, $i<j$,
     w.r.t.~it. In the second pass from Step 9 to 13 we take each
     $g_i$, $i=1,\ldots,k-1$, and reduce it w.r.t.~all $g_j$, $i<j$.

     Now suppose we enter the Algorithm with $H\cup\{g\}$, where
     $H=\{h_1,\ldots,h_k\}$ is already initially reduced w.r.t.~itself and $p-t$. Suppose furthermore
     \begin{displaymath}
       \lm_>(h_1)> \ldots > \lm_>(h_l)>\lm_>(g)>\lm_>(h_{l+1}) > \ldots > \lm_>(h_k).
     \end{displaymath}

     By assumption, taking each $h_i$, $i=1,\ldots,l$, and reducing all
     $h_j$, $i<j$, w.r.t.~it is obsolete. The first necessary
     action is reducing $g$ w.r.t.~$h_1,\ldots,h_l$.
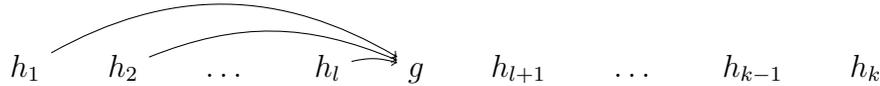
     \begin{center}\begin{tikzpicture}[descr/.style={fill=white,inner sep=2pt}]
         \matrix (m) [matrix of math nodes, row sep=2em,
         column sep=1.5em, text height=1.5ex, text depth=0.25ex]
         { h_1 & h_2 &\ldots & h_l & g & h_{l+1} & \ldots & h_{k-1} & h_k \\};
         \path[->,font=\scriptsize]
         (m-1-1) edge [bend left = 30] (m-1-5)
         (m-1-2) edge [bend left = 22.5] (m-1-5)
         (m-1-4) edge [bend left = 15] (m-1-5);
       \end{tikzpicture}\end{center}

     Next, we consider the $h_i$, $i=l+1,\ldots,k$. Each $h_i$ is already
     reduced w.r.t.~$h_1,\ldots,h_l$ and remains so after
     reducing it w.r.t.~$g$, as the ${\ux}$-monomials of their
     leading monomials were already completely eliminated in $g$
     previously. Hence we may reduce each $h_i$, $i=l+1,\ldots,k$,
     w.r.t.~$g$ without inducing the need of reducing them w.r.t.~$h_1,\ldots,h_l$ again.
     
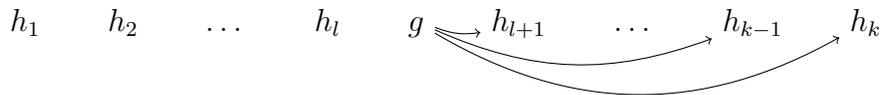
\begin{center}\begin{tikzpicture}[descr/.style={fill=white,inner sep=2pt}]
         \matrix (m) [matrix of math nodes, row sep=2em,
         column sep=1.5em, text height=1.5ex, text depth=0.25ex]
         { h_1 & h_2 &\ldots & h_l & g & h_{l+1} & \ldots & h_{k-1} & h_k \\};
         \path[->,font=\scriptsize]
         (m-1-5) edge [bend right = 30] (m-1-9)
         (m-1-5) edge [bend right = 22.5] (m-1-8)
         (m-1-5) edge [bend right = 15] (m-1-6);
       \end{tikzpicture}\end{center}

     However, $g$ might contain a term with monomial $t^2 x$, which might
     not be reducible w.r.t.~$\lt_>(h_j)=t^3 x$, but if $g$ is
     multiplied by $t$ while reducing another element w.r.t.~it,
     we do create a term that is reducible. Thus, we need to reduce each
     $h_i$, $i=l+1,\ldots,k$ w.r.t.~$h_j$, $j=l+1,\ldots,i$
     again and this concludes our first pass.
     
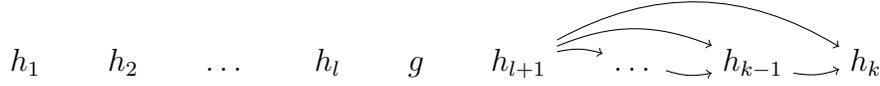
\begin{center}\begin{tikzpicture}[descr/.style={fill=white,inner sep=2pt}]
         \matrix (m) [matrix of math nodes, row sep=2em,
         column sep=1.5em, text height=1.5ex, text depth=0.25ex]
         { h_1 & h_2 &\ldots & h_l & g & h_{l+1} & \ldots & h_{k-1} & h_k \\};
         \path[->,font=\scriptsize]
         (m-1-6) edge [bend left = 30] (m-1-9)
         (m-1-6) edge [bend left = 22.5] (m-1-8)
         (m-1-6) edge [bend left = 15] (m-1-7);
         \path[->,font=\scriptsize]
         (m-1-7) edge [bend right = 15] (m-1-8)
         (m-1-8) edge [bend right = 15] (m-1-9);
       \end{tikzpicture}\end{center}

     For the second pass, taking each $h_i$, $i=1,\ldots,l$, and reducing
     it w.r.t.~all $h_j$, $j=i+1,\ldots,l$, is unnecessary. The
     first necessary step is to take each $h_i$, $i=1,\ldots,l$, and
     reduce it w.r.t.~the newly added $g$. Similar to a previous
     step, each $h_i$ remains reduced w.r.t.~all $h_j$,
     $j=i+1,\ldots,l$.
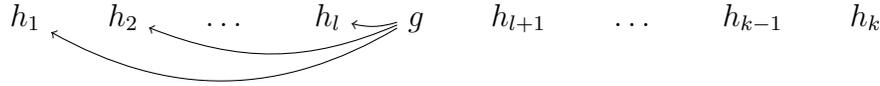
     \begin{center}\begin{tikzpicture}[descr/.style={fill=white,inner sep=2pt}]
         \matrix (m) [matrix of math nodes, row sep=2em,
         column sep=1.5em, text height=1.5ex, text depth=0.25ex]
         { h_1 & h_2 &\ldots & h_l & g & h_{l+1} & \ldots & h_{k-1} & h_k \\};
         \path[->,font=\scriptsize]
         (m-1-5) edge [bend left = 30] (m-1-1)
         (m-1-5) edge [bend left = 22.5] (m-1-2)
         (m-1-5) edge [bend left = 15] (m-1-4);
       \end{tikzpicture}\end{center}

     Afterwards, while each $h_i$ remains reduced w.r.t.~all
     $h_j$, $j=i+1,\ldots,l$, it nonetheless needs to be reduced w.r.t.~$h_{l+1},\ldots,h_k$ again.
     
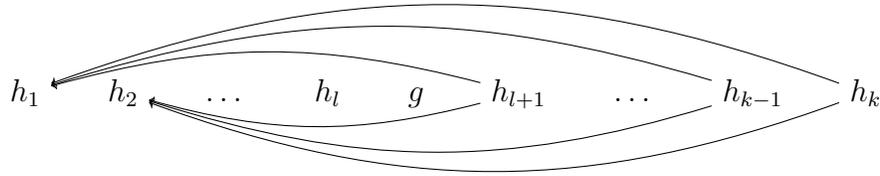
\begin{center}\begin{tikzpicture}[descr/.style={fill=white,inner sep=2pt}]
         \matrix (m) [matrix of math nodes, row sep=2em,
         column sep=1.5em, text height=1.5ex, text depth=0.25ex]
         { h_1 & h_2 &\ldots & h_l & g & h_{l+1} & \ldots & h_{k-1} & h_k \\};
         \path[->,font=\scriptsize]
         (m-1-6) edge [bend right = 15] (m-1-1)
         (m-1-8) edge [bend right = 17.5] (m-1-1)
         (m-1-9) edge [bend right = 20] (m-1-1);
         \path[->,font=\scriptsize]
         (m-1-6) edge [bend left = 15] (m-1-2)
         (m-1-8) edge [bend left = 17.5] (m-1-2)
         (m-1-9) edge [bend left = 20] (m-1-2);
       \end{tikzpicture}\end{center}

     Next in the second pass, we take $g$ and reduce it w.r.t.~$h_{l+1},\ldots,h_k$.
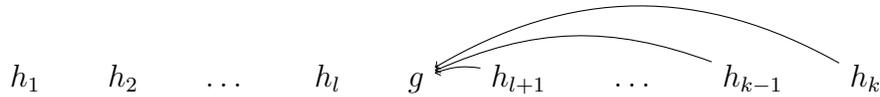
     \begin{center}\begin{tikzpicture}[descr/.style={fill=white,inner sep=2pt}]
         \matrix (m) [matrix of math nodes, row sep=2em,
         column sep=1.5em, text height=1.5ex, text depth=0.25ex]
         { h_1 & h_2 &\ldots & h_l & g & h_{l+1} & \ldots & h_{k-1} & h_k \\};
         \path[->,font=\scriptsize]
         (m-1-9) edge [bend right = 30] (m-1-5)
         (m-1-8) edge [bend right = 22.5] (m-1-5)
         (m-1-6) edge [bend right = 15] (m-1-5);
       \end{tikzpicture}\end{center}

     And finally, we take each $h_i$, $i=l+1,\ldots,k-1$ and reduce it
     w.r.t.~all $h_j$, $i<j$, as reducing them w.r.t.~$g$ earlier
     might have broken their reducedness property.
     
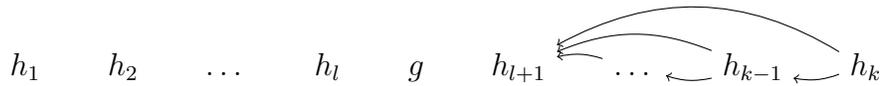
\begin{center}\begin{tikzpicture}[descr/.style={fill=white,inner sep=2pt}]
         \matrix (m) [matrix of math nodes, row sep=2em,
         column sep=1.5em, text height=1.5ex, text depth=0.25ex]
         { h_1 & h_2 &\ldots & h_l & g & h_{l+1} & \ldots & h_{k-1} & h_k \\};
         \path[->,font=\scriptsize]
         (m-1-9) edge [bend right = 30] (m-1-6)
         (m-1-8) edge [bend right = 22.5] (m-1-6)
         (m-1-7) edge [bend right = 15] (m-1-6);
         \path[->,font=\scriptsize]
         (m-1-9) edge [bend left = 15] (m-1-8)
         (m-1-8) edge [bend left = 15] (m-1-7);
       \end{tikzpicture}\end{center}

     It can be seen that a position of $g$ more to the right minimises
     the number of reductions needed. This implies that $\lm_>(g)$ should
     be as small as possible, and since its monomial in $x$ is fixed,
     this means that it should have as high a degree in $t$ as possible.

     Note that increasing the degree in $t$ to increase performance is
     not risk-free a priori. For example, suppose we had a $g\in G$ with
     $\lt_>(g)=x$ and we were to add $t^5y\cdot g$ to $E$ in order to
     reduce a term with monomial $t^5xy$. Then any subsequent term with
     monomial $t^4xy$ would require adding an additional multiple of $g$
     to $E$. However, since our working list $T$ is worked off in an
     order induced by a $t$-local monomial ordering $>$, any later $s'$
     picked in Step~4 with the same monomial in $x$ necessarily has to
     have a higher degree in $t$. Thus this cannot happen in our
     algorithm.
   \end{remark}

   With Algorithm\lang{s~\ref{alg:inred02} and}~\ref{alg:inred01}, writing an
   algorithm for computing an initially reduced standard basis becomes a
   straightforward task. All we need to adhere is to proceed ${\ux}$-degree
   by ${\ux}$-degree while repeatedly applying the previous algorithm.

   \begin{algorithm}[initially reduced standard basis]\label{alg:inred} \ \vspace*{-3ex}
     \begin{algorithmic}[1]
       \REQUIRE{$(F,>)$, where $F\subset I$ an ${\ux}$-homogeneous, polynomial generating set of $I$ containing $p-t$.}
       \ENSURE{$G\subseteq I$ an ${\ux}$-homogeneous, polynomial and initially reduced standard basis of $I$.}
       \STATE Compute an ${\ux}$-homogeneous standard basis $G''$ of
       $I=\langle F \rangle$ with \cite[Alg.~2.16 or Alg.~3.8]{MRW15}.
       \STATE Set $G':=\emptyset$.
       \FOR{$g\in G''$ with $p\nmid \lt_>(g)$}
       \IF{$\lc_>(g)\neq 1$}
       \STATE Since $1\in \langle \lc_>(g),p\rangle$, find $a,b\in R$ such that
       \begin{displaymath}
         1=a\cdot\lc_>(g)+b\cdot p.
       \end{displaymath}
       \STATE Set
       \begin{displaymath}
         g:=a\cdot g + b\cdot\lm_>(g)\cdot(p-t),
       \end{displaymath}
       so that $\lc_>(g)=1$.
       \ENDIF
       \STATE Set $G':=G'\cup\{g\}$.
       \ENDFOR
       \STATE Minimise the standard basis $G'$ by gradually removing elements $g\in G$ with $\lm_>(g')\mid \lm_>(g)$ for some $g'\in G$, $g'\neq g$.
       \STATE Set $G:=\emptyset$
       \WHILE{$G'\neq \emptyset$}
       \STATE Set
       \begin{align*}
         d&:=\min\{\deg_{\ux}(g)\mid g\in G'\}, \\
         H'&:=\{g\in G'\mid \deg_{\ux}(g)=d \}, \\
         G'&:=\{g\in G'\mid \deg_{\ux}(g)>d \}.
       \end{align*}
       \STATE Reduce $H'$ initially w.r.t.~$G$, $p-t$ and itself using
       Algorithm\lang{s~\ref{alg:inred02} or}~\ref{alg:inred01} and let $H$ be the output of that initial reduction.
       \STATE Set $G:=G\cup H$.
       \ENDWHILE
       \RETURN{$G\cup \{p-t\}$.}
     \end{algorithmic}
   \end{algorithm}
   \begin{proof}
     It is clear that $G$ is a standard basis of $I$, as we are merely normalising the leading coefficients of the standard bases $G''$.
     It is also obvious that $G$ is polynomial and ${\ux}$-homogeneous.
     The initial reducedness of $G$ follows from the correctness of Algorithm\lang{s~\ref{alg:inred02} or}~\ref{alg:inred01}.
   \end{proof}


   \section{How to compute the Gr\"obner fan}\label{sec:computation}

   In this section, we describe algorithms for computing the Gr\"obner
   fan of an ideal $I\unlhd\Rtx$ as in our convention on
   Page~\pageref{page:basering}, provided that we are able to compute
   initially reduced standard bases where needed. While computing a Gr\"obner fan can
   be as seemingly simple as computing maximal Gr\"obner cones $C_>(I)$
   w.r.t.~random monomial orderings $>$ until the whole weight
   space $\RR_{\leq 0}\times\RR^n$ is filled, sensible algorithms
   avoid computing initially reduced standard bases of
   $I$ from scratch. The algorithms in this section are adjusted versions
   of the algorithms found in Chapter $4$ of Jensen's dissertation
   \cite{Jensen07} (see also \cite{FJT07}), though some of the ideas involved originate in
   Collart, Kalkbrenner and Mall's work on the Gr\"obner walk
   \cite{CKM97}.

   We start with an algorithm for computing witnesses of weighted
   homogeneous elements in initial ideals, which can then be used to lift
   standard bases of initial ideals to initially reduced standard bases
   of the original ideal. Adding in some statements about the
   perturbation of initial ideals, we obtain an algorithm which allows us
   to flip initially reduced standard bases of one ordering to initially
   reduced standard bases of an adjacent ordering. This algorithm can
   then be used to construct the Gr\"obner fan, requiring us to compute
   the standard basis of $I$ from scratch only once.

   Note that all polynomial computations in our algorithms,
   if given polynomial input, terminate and return polynomial output
   themselves, provided that we are able to initially reduce a
   standard basis as e.g.~in Algorithm~\ref{alg:inred}.

   \begin{algorithm}[Witness]\label{alg:witness} \ \vspace*{-3ex}
     \begin{algorithmic}[1]
       \REQUIRE{$(h,H,G,>)$, where
         \begin{itemize}[leftmargin=*, itemsep=0pt]
         \item $>$ a weighted $t$-local monomial ordering on $\Mon(t,{\ux})$,
         \item $G=\{g_1,\ldots,g_k\}\subseteq I$ an initially reduced standard basis of $I$ w.r.t.~$>$,
         \item $H=\{h_1,\ldots,h_k\}$ with
           $h_i=\initial_w(g_i)$ for some $w\in C_>(I)$ with $w_0<0$,
         \item $h\in\initial_w(I)$ weighted homogeneous w.r.t.~$w$.
         \end{itemize}}
       \ENSURE{$f\in I$ such that $\initial_w(f)=h$}
       \STATE Use \cite[Alg.~1.13]{MRW15} to compute a homogeneous determinate division with remainder w.r.t.~$>$,
       \begin{displaymath}
         (\{q_1,\ldots,q_k\},r)=\HDDwR(h,\{h_1,\ldots,h_k\},>),
       \end{displaymath}
       so that $h = q_1\cdot h_1 + \hdots + q_k\cdot h_k$ and $r=0$.
       \RETURN{$f:=q_1\cdot g_1+ \hdots + q_k\cdot g_k$}
     \end{algorithmic}
   \end{algorithm}
   \begin{proof}
     By Proposition~\ref{prop:BasisOfInitialIdeal}, $H$
     is a standard basis of $\initial_w(I)$, therefore the division of $h$ will always yield remainder $0$.

     Since $h,h_1,\ldots,h_k$ are weighted homogeneous w.r.t.~$w$,
     so are $q_1,\ldots,q_k$.
     Hence
     \begin{displaymath}
       \initial_w(f)=\underbrace{\initial_w(q_1)\cdot\initial_w(g_1)}_{=q_1\cdot \,h_1}+\hdots
       +\underbrace{\initial_w(q_k)\cdot\initial_w(g_k)}_{=q_k\cdot \,h_k} = h.
     \end{displaymath}
     Also note that the division with remainder will always terminate, as the weighted degree
     cannot become arbitrarily small since the ideal $\initial_w(I)$
     is homogeneous in ${\ux}$ and weighted homogeneous overall.
   \end{proof}

   As announced, we immediately obtain an algorithm which allows us to
   lift a standard basis of an initial ideal to an initially reduced
   standard basis of $I$, assuming we have a standard basis of $I$
   w.r.t.~an adjacent ordering at our disposal. 

   \begin{algorithm}[Lift]\label{alg:lift} \ \vspace*{-3ex}
     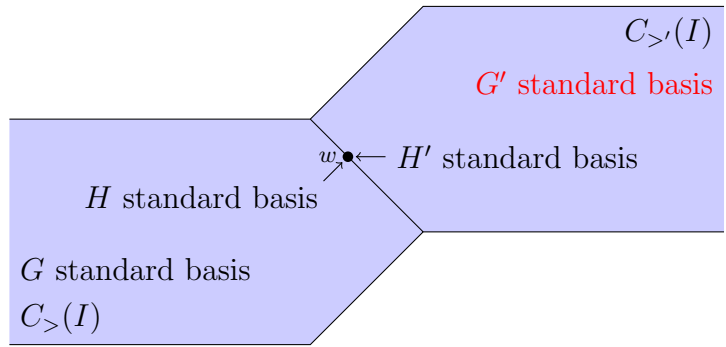
\begin{figure}[h]
       \centering
       \begin{tikzpicture}
         \fill[color=blue!20] (0,0) -- (4,0) -- (5.5,-1.5) -- (4,-3) -- (0,-3) -- cycle;
         \fill[color=blue!20] (4,0) -- (5.5,1.5) -- (9.5,1.5) -- (9.5,-1.5) -- (5.5,-1.5) -- cycle;
         \draw (0,0) -- (4,0) -- (5.5,-1.5) -- (4,-3) -- (0,-3);
         \draw (4,0) -- (5.5,1.5) -- (9.5,1.5);
         \draw (5.5,-1.5) -- (9.5,-1.5);
         \node[anchor=south west] (C1) at (0,-3) {$C_>(I)$};
         \node[anchor=south west, yshift=0cm] at (C1.north west) {$G$ standard basis};
         \node[anchor=north east, xshift=-0.25cm, yshift=-0.25cm] (H1) at (4.5,-0.5) {$H$ standard basis};
         \node[anchor=north east] (C2) at (9.5,1.5) {$C_{>'}(I)$};
         \node[anchor=north east, red] at (C2.south east) {$G'$ standard basis};
         \node[anchor=west, xshift=0.5cm] (H2) at (4.5,-0.5) {$H'$ standard basis};
         \fill (4.5,-0.5) circle (2pt);
         \node[anchor=east, font=\scriptsize] at (4.5,-0.5) {$w$};

         \draw[->, shorten >= 3pt, shorten <= -3pt] (H1.north east) -- (4.5,-0.5);
         \draw[->, shorten >= 3pt] (H2.west) -- (4.5,-0.5);
       \end{tikzpicture}
       \caption{lift of standard bases}
       \label{fig:liftOfStandardBases}
     \end{figure}
     \begin{algorithmic}[1]
       \REQUIRE{$(H',>',H,G,>)$, where
         \begin{itemize}[leftmargin=*, itemsep=0pt]
         \item $>$ a weighted $t$-local monomial ordering on $\Mon(t,{\ux})$ with weight vector in $\RR_{<0}\times\RR^n$,
         \item $G=\{g_1,\ldots,g_k\}\subseteq I$ an initially reduced standard basis of $I$ w.r.t.~$>$,
         \item $H=\{h_1,\ldots,h_k\}$ with $h_i=\initial_w(g_i)$ for
           some $w\in C_>(I)$ with $w_0<0$,
         \item $>'$ a $t$-local monomial ordering such that $w\in C_>(I)\cap C_{>'}(I)$,
         \item $H'\subseteq \initial_w(I)$ a weighted homogeneous standard basis w.r.t.~$>'$.
         \end{itemize}}
       \ENSURE{$G'\subseteq I$, an initially reduced standard basis of $I$ w.r.t.~$>'$.
       }
       \STATE Set $G'':=\{\Witness(h, H, G, >)\mid h\in H'\}$.
       \STATE Reduce $G''$ initially
       w.r.t.~$>'$ and obtain $G'$.
       \RETURN{$G'$.}
     \end{algorithmic}
   \end{algorithm}
   \begin{proof}
     Consider a witness $g:=\Witness(h, w, G, >)$ for some $h\in H'$.
     Then, by Lemma~\ref{lem:Cmembership}, we have $\lt_{>'}(g)=\lt_{>'}(\initial_w(g))=\lt_{>'}(h)$,
     and thus
     \begin{displaymath}
       \langle \lt_{>'}(g) \mid g\in G'' \rangle = \langle \lt_{>'}(h) \mid h\in H' \rangle = \lt_{>'}(\initial_w(I))
       \underset{\ref{lem:LeadOfInitialIdeal}}{\overset{\text{Lem.}}{=}} \lt_{>'}(I).
     \end{displaymath}
     Thus $G''$ is a standard basis of $I$ w.r.t.~$>'$ and $G'$ is even
     initially reduced.
   \end{proof}

   \begin{example}
     Consider again the ideal from Example~\ref{ex:31}
     \begin{displaymath}
       I=\langle g_1=x-t^3x+t^3z-t^4z,g_2=y-t^3y+t^2z-t^4z\rangle\unlhd \ZZ\llbracket t\rrbracket [x,y,z]
     \end{displaymath}
     and the weighted monomial ordering $>=>_v$ on $\Mon(t,x,y,z)$ with
     weight vector $v=(-1,3,3,3)\in\RR_{<0}\times\RR^3$ and the $t$-local
     lexicographical ordering such that $x>y>z>1>t$ as tiebreaker.
     We have already seen that
     \begin{displaymath}
       C_>(I)= \overline{\{w\in\RR_{<0}\times\RR^n\mid w_1 \geq 3w_0+w_3 \text{ and } w_2\geq 2w_0+w_3 \}}.
     \end{displaymath}
     Picking $w=(-1,2,-1,1)$ in a facet of $C_>(I)$, Proposition~\ref{prop:BasisOfInitialIdeal} implies
     \begin{displaymath}
       \initial_w(I) = \langle \initial_w(g_1), \initial_w(g_2) \rangle = \langle x, y+t^2z \rangle.
     \end{displaymath}
     It is easy to see that $\{x, y+t^2z\}$ is a standard basis of
     $\initial_w(I)$ regardless which monomial ordering is chosen. Since
     using Algorithm~\ref{alg:witness} on $\initial_w(g_1)$ and
     $\initial_w(g_2)$ yields $g_1$ and $g_2$ respectively, Algorithm~\ref{alg:lift} therefore implies that $\{g_1,g_2\}$ is also a
     standard basis for the adjacent monomial ordering $>'$ on the other
     side of the facet containing $w$.

     Moreover, since $>'$ has to induce a different leading ideal by
     definition, and the leading terms of $g_1$ and $g_2$ w.r.t.~$>'$
     have to occur in their initial forms by
     Lemma~\ref{lem:Cmembership}, we see that the adjacent leading
     ideal is  $\langle x, t^2z\rangle$.
   \end{example}

   An easy way to construct orderings adjacent to $>$ is by connecting
   two weight vectors in series, the first a weight vector lying on a
   facet and the second an outer normal vector of the facet.

   \begin{proposition}\label{prop:PertubationOfInitialIdeal}
     Let $>$ be a $t$-local monomial ordering, $w \in C_>(I)$ with $w_0<0$ and
     $v\in\RR^{n+1}$. Let $>_{(w,v)}$ denote the $t$-local monomial
     ordering given by
     \begin{align*}
       & t^\beta \cdot  {\ux}^\alpha >_{(w,v)} t^{\beta'} \cdot  {\ux}^{\alpha'} \quad :\Longleftrightarrow \quad \\
       & \qquad (\beta,\alpha)\cdot w > (\beta',\alpha')\cdot w, \\
       & \qquad \text{or } (\beta,\alpha)\cdot w = (\beta',\alpha')\cdot w \text{ and }
       (\beta,\alpha)\cdot v > (\beta',\alpha')\cdot v, \\
       & \qquad \text{or } (\beta,\alpha)\cdot w = (\beta',\alpha')\cdot w \text{ and }
       (\beta,\alpha)\cdot v = (\beta',\alpha')\cdot v \\
       & \qquad \qquad \qquad    \text{  and } t^\beta \cdot  {\ux}^\alpha > t^{\beta'} \cdot  {\ux}^{\alpha'}.
     \end{align*}
     Then $w = C_>(I)\cap C_{>_{(w,v)}}(I)$ and for $\varepsilon>0$
     sufficiently small
     \begin{displaymath}
       w+\varepsilon \cdot v\in C_{>_{(w,v)}}(I).
     \end{displaymath}
     In particular for these $\varepsilon$ we have
     \begin{math}
       \initial_{w+\varepsilon v}(I) = \initial_v(\initial_w(I)).
     \end{math}
   \end{proposition}
   \begin{proof}
     By definition we have
     $\lt_{>_{(w,v)}}(g)=\lt_{>_{(w,v)}}(\initial_w(g))$ for any
     $g\in\Rtx$, which implies $w\in C_{>_{(w,v)}}(I)$ by Lemma~\ref{lem:Cmembership}.

     Next, let $G$ be an initially reduced standard basis of $I$ w.r.t.~that ordering.
     Observe that every $g\in G$,
     \begin{displaymath}
       g = \underbrace{\hdots \hdots \hdots \hdots \hdots}_{\initial_w(g)}
       + \underbrace{\hdots \hdots \hdots \hdots \hdots}_{\text{rest}}\;,
     \end{displaymath}
     has a distinct degree gap between the terms of highest weighted
     degree and the rest. As the weighted degree varies continuously
     under the weight vector, choosing $\varepsilon > 0$ sufficiently
     small ensures that the $(w+\varepsilon\cdot v)$-weighted degrees of
     the terms in $\initial_w(g)$ remain higher than those of the rest.
     Thus $\initial_{w+\varepsilon\cdot v}(g)$ is the sum of those
     terms of $\initial_w(g)$ that have maximal $v$-weighted degree,
     i.e. $\initial_{w+\varepsilon\cdot v}(g) = \initial_v(\initial_w(g))$.
     In particular, we have
     \begin{displaymath}
       \lt_{>_{(w,v)}}(\initial_{w+\varepsilon\cdot v}(g))=\lt_{>_{(w,v)}}(g),
     \end{displaymath}
     and hence $w+\varepsilon\cdot v \in C_{>_{(w,v)}}(I)$ by Lemma~\ref{lem:Cmembership} again.

     The final claim now follows from Proposition~\ref{prop:BasisOfInitialIdeal}:
     \begin{displaymath}
       \initial_{w+\varepsilon\cdot v}(I)
       \underset{\ref{prop:BasisOfInitialIdeal}}{\overset{\text{Prop.}}{=}}
       \langle \initial_{w+\varepsilon\cdot v}(g)\mid g\in G \rangle
       = \langle \initial_v(\initial_w(g))\mid g\in G \rangle
       \underset{\ref{prop:BasisOfInitialIdeal}}{\overset{\text{Prop.}}{=}}
       \initial_v(\initial_w(I)).
     \end{displaymath}
   \end{proof}

   With this easy method of constructing adjacent orderings, we are now
   able to write an algorithm for flipping initially reduced standard
   bases.

   \begin{algorithm}[Flip]\label{alg:flip} \ \vspace*{-3ex}
     \begin{figure}[h]
       \centering
       \begin{tikzpicture}
         \fill[color=blue!20] (0,0) -- (4,0) -- (5.5,-1.5) -- (4,-3) -- (0,-3) -- cycle;
         \fill[color=blue!20] (4,0) -- (5.5,1.5) -- (9.5,1.5) -- (9.5,-1.5) -- (5.5,-1.5) -- cycle;
         \draw (0,0) -- (4,0) -- (5.5,-1.5) -- (4,-3) -- (0,-3);
         \draw (4,0) -- (5.5,1.5) -- (9.5,1.5);
         \draw (5.5,-1.5) -- (9.5,-1.5);
         \node[anchor=south west] (C1) at (0,-3) {$C_>(I)$};
         \node[anchor=south west, yshift=0cm] at (C1.north west) {$G$ standard basis};
         \node[anchor=north east] (C2) at (9.5,1.5) {$C_{>'}(I)$};
         \node[anchor=north east, red] at (C2.south east) {$G'$ standard basis};
         \fill (4.5,-0.5) circle (2pt);
         \node[anchor=east] at (4.5,-0.5) {$w$};
         \draw[->] (4.5,-0.5) -- (5,0);
         \node[anchor=south west] at (5,0) {$v$};
         \node[anchor=north east, xshift=-0.25cm, yshift=-0.25cm] (H1) at (4.5,-0.5) {$H$ standard basis};
         \draw[->, shorten >= 3pt, shorten <= -3pt] (H1.north east) -- (4.5,-0.5);
       \end{tikzpicture}
       \caption{flip of standard bases}
       \label{fig:flipOfStandardBases}
     \end{figure}
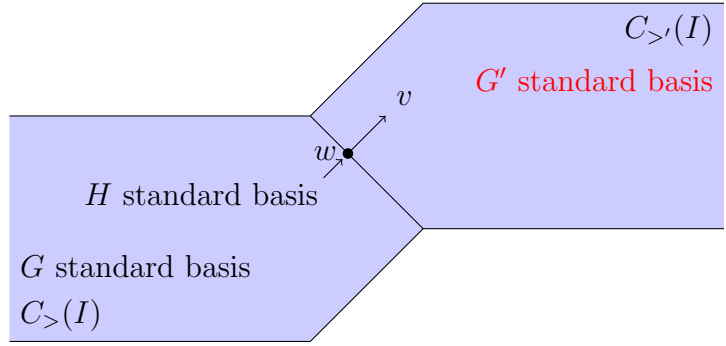
     \begin{algorithmic}[1]
       \REQUIRE{$(G,H,v,>)$, where
         \begin{itemize}[leftmargin=*, itemsep=0pt]
         \item $>$ a weighted $t$-local monomial ordering on $\Mon(t,{\ux})$ with weight vector in $\RR_{<0}\times\RR^n$,
         \item $G=\{g_1,\ldots,g_k\}\subseteq I$ an initially reduced standard basis of $I$ w.r.t.~$>$,
         \item $H=\{h_1,\ldots,h_k\}$ with $h_i=\initial_w(g_i)$ for
           some  relative interior point $w\in C_>(I)$ on a lower
           facet $\tau\leq C_>(I)$, $\tau\nsubseteq \{0\}\times\RR^n$
           and $w_0<0$.
         \item $v\in \RR\times \RR^n$ an outer normal vector of the facet $\tau$.
         \end{itemize}}
       \ENSURE{$(G',>')$, where $>'$ is an adjacent $t$-local monomial ordering with
         \begin{displaymath}
           \tau = C_>(I)\cap C_{>'}(I) \quad \text{and} \quad C_>(I)\neq C_{>'}(I),
         \end{displaymath}
         and $G'\subseteq I$ is an initially reduced standard basis w.r.t.~$>'$.}
       \STATE Compute a standard basis $H'$ of $\langle H\rangle=\initial_w(I)$ w.r.t.~$>_{(w,v)}$.
       \STATE Set $G':= \Lift(H',>_{(w,v)},H,G,>)$.
       \RETURN{$(G',>_{(w,v)})$}
     \end{algorithmic}
   \end{algorithm}
   \begin{proof}
     By our Lifting Algorithm~\ref{alg:lift}, $G'$ is an initially
     reduced standard basis of $I$ w.r.t.~$>_{(w,v)}$. The remaining
     conditions follow from
     Proposition~\ref{prop:PertubationOfInitialIdeal}.
   \end{proof}

   \begin{example}
     Consider the ideal
     \begin{displaymath}
       I:=\langle 2-t, xy^2-t^2y^3, x^2-t^3y^2 \rangle \unlhd \ZZ\llbracket t \rrbracket [x,y]
     \end{displaymath}
     and the weighted monomial ordering $>=>_u$ on $\Mon(t,x,y)$ with
     weight vector $u:=(-1,1,1)\in\RR_{<0}\times\RR^2$ and $t$-local
     lexicographical ordering such that $x>y>1>t$ as tiebreaker. An
     initially reduced standard basis of $I$ is then given by
     \begin{displaymath}
       G:=\{ 2-t, xy^2-t^2y^3, x^2-t^3y^2, t^3y^4\}.
     \end{displaymath}
     The maximal Gr\"obner cone $C_>(I)\subseteq\RR_{\leq 0}\times\RR^2$ is determined by the inequalities
     \begin{align*}
       (w_t,w_x,w_y)\in C_>(I) \quad &\Longleftrightarrow \quad \begin{cases} w_x+2w_y\geq 2w_t+3w_y \\ 2w_x\geq 3w_t+2w_y \end{cases} \\
       &\Longleftrightarrow \quad \begin{cases} w_x\geq 2w_t+w_y \\ 2w_x\geq 3w_t+2w_y \end{cases}
     \end{align*}
     It is easy to see how $w:=(-4,1,7)$ is contained in $C_>(I)$. In
     fact, it lies on its boundary since $2w_x=3w_t+2w_y=2$. Then
     $v:=(3,5,1)\in\RR^3$ is an outer normal vector, as even for small
     $\varepsilon>0$
     \begin{displaymath}
       \underbrace{2(w_x+\varepsilon \cdot v_x)}_{2+10\varepsilon}\ngeq \underbrace{3(w_t+\varepsilon \cdot v_t)}_{-12+9\varepsilon}+\underbrace{2(w_y+\varepsilon \cdot v_y)}_{14+2\varepsilon}.
     \end{displaymath}
     An initially reduced standard basis of $\initial_w(I)$ is then given by
     \begin{displaymath}
       H:=\{\initial_w(g)\mid g\in G\} = \{2, xy^2, x^2-t^3y^2, t^3y^4\},
     \end{displaymath}
     and computing a standard basis of $\initial_w(I)$ w.r.t.~the ordering $>_{(w,v)}$ yields
     \begin{displaymath}
       H':=\{2,xy^2,t^3y^2-x^2,x^3\},
     \end{displaymath}
     which can then be lifted to a standard basis of $I$ w.r.t.~the same ordering $>_{(w,v)}$ that is adjacent to $>$
     \begin{displaymath}
       G'=\{2-t,xy^2-t^2y^3,t^3y^2-x^2,x^3-t^5y^3\}.
     \end{displaymath}
   \end{example}

   The Gr\"obner fan algorithm is a so-called fan traversal algorithm. We
   start with computing a starting cone and repeatedly use
   Algorithm~\ref{alg:flip} to compute adjacent cones until we obtain
   the whole
   fan. The whole process is commonly illustrated on a bipartite graph as
   shown in Figure~\ref{fig:groebnerFanBipartiteGraph}. This bipartite
   graph also satisfies the so-called reverse search property, which can
   be used for further optimisation. See Chapter 3.2 in \cite{Jensen07}
   for more information about the reverse search property of Gr\"obner
   fans.

   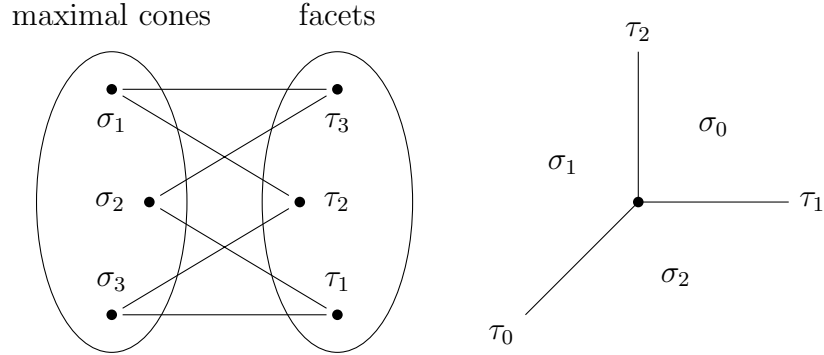
\begin{figure}[h]
     \centering
     \begin{tikzpicture}
       \node (s1) at (0,1.5) {};
       \node (s2) at (0.5,0) {};
       \node (s3) at (0,-1.5) {};
       \node (t3) at (3,1.5) {};
       \node (t2) at (2.5,0) {};
       \node (t1) at (3,-1.5) {};
       \fill (s1) circle (2pt);
       \fill (s2) circle (2pt);
       \fill (s3) circle (2pt);
       \fill (t3) circle (2pt);
       \fill (t2) circle (2pt);
       \fill (t1) circle (2pt);
       \node[anchor=north,yshift=-5pt] at (s1) {$\sigma_1$};
       \node[anchor=east,xshift=-5pt] at (s2) {$\sigma_2$};
       \node[anchor=south,yshift=5pt] at (s3) {$\sigma_3$};
       \node[anchor=north,yshift=-5pt] at (t3) {$\tau_3$};
       \node[anchor=west,xshift=5pt] at (t2) {$\tau_2$};
       \node[anchor=south,yshift=5pt] at (t1) {$\tau_1$};
       \draw (s1) -- (t2)
       (s1) -- (t3)
       (s2) -- (t1)
       (s2) -- (t3)
       (s3) -- (t1)
       (s3) -- (t2);
       \draw (0,0) ellipse (1cm and 2cm);
       \draw (3,0) ellipse (1cm and 2cm);
       \node at (0,2.5) {maximal cones};
       \node at (3,2.5) {facets};

       \draw (7,0) -- (9,0)
       (7,0) -- (7,2)
       (7,0) -- (5.5,-1.5);
       \fill (7,0) circle (2pt);
       \node at (8,1) {$\sigma_0$};
       \node at (6,0.5) {$\sigma_1$};
       \node at (7.5,-1) {$\sigma_2$};
       \node[anchor=west] at (9,0) {$\tau_1$};
       \node[anchor=south] at (7,2) {$\tau_2$};
       \node[anchor=north east] at (5.5,-1.5) {$\tau_0$};
     \end{tikzpicture}%
     \caption{The bipartite graph of a Gr\"obner fan $\Sigma(\langle x+y+z\rangle)$}
     \label{fig:groebnerFanBipartiteGraph}
   \end{figure}

   Note that since the Gr\"obner fan spans the whole weight space
   $\RR_{\leq 0}\times\RR^n$, each lower facet is contained in exactly two
   maximal cones. That means, traversing a facet $\tau \leq C_>(I)$ of a
   Gr\"obner cone $C_>(I)$ can be omitted if $\tau$ is contained
   in a maximal Gr\"obner cone that was already computed.

   \begin{algorithm}[Gr\"obner fan]\label{alg:groebnerFan} \ \vspace*{-3ex}
     \begin{algorithmic}[1]
       \REQUIRE{$F\subseteq I\unlhd\Rtx$ an ${\ux}$-homogeneous generating set.}
       \ENSURE{$\Delta$, the set of maximal cones in the Gr\"obner fan $\Sigma(I)$ of $I$.}
       \STATE Pick a random weight $u\in\RR_{<0}\times\RR^n$ and a $t$-local monomial ordering $>$.
       \STATE Compute an initially reduced standard basis $G$ of $I$ w.r.t.~$>_u$.
       \STATE Construct the maximal Gr\"obner cone
       $C_{>_u}(I)=C(\lt_{>_u}(G),G,>_u)$ using Algorithm~\ref{alg:groebnerConeNoWeight}.
       \STATE Initialise the Gr\"obner fan $\Sigma:=\{C_{>_u}(I)\}$.
       \STATE Initialise a working list $L:=\{(G,>_u,C_{>_u}(I))\}$.
       \WHILE{$L\neq \emptyset$}
       \STATE Pick $(G,>_u,C_{>_u}(I))\in L$.
       \FOR{all facets $\tau\leq C_{>_u}(I)$, $\tau\nsubseteq \{0\}\times\RR^n$}
       \STATE Compute a relative interior point $w\in\tau$.
       \IF{$w\notin C_{>'}(I)$ for all $C_{>'}(I)\in\Sigma\setminus \{C_{>_u}(I)\}$}
       \STATE Compute an outer normal vector $v$ of $\tau$.
       \STATE Set $H:=\{ \initial_w(g)\mid g\in G \}$.
       \STATE Compute $(G',>'):=\Flip(G,H,v,>_u)$ using Algorithm~\ref{alg:flip}.
       \STATE Construct the adjacent  cone
       $C_{>'}(I)=C(\lt_{>'}(G'),G',>')$.
       \STATE Compute a relative interior point $u'\in C_{>'}(I)$, so that $G'$ is a standard basis w.r.t.~$>_{u'}$ and $C_{>'}(I)=C_{>_{u'}}(I)$.
       \STATE Set $\Sigma:=\Sigma\cup\{C_{>_u'}(I)\}$.
       \STATE Set $L:=L\cup\{ (G',>_{u'},C_{>_{u'}}(I)\}$.
       \ENDIF
       \ENDFOR
       \STATE Set $L:=L\setminus\{(G,>_u,C_{>_u}(I))\}$.
       \ENDWHILE
       \RETURN{$\Delta$}
     \end{algorithmic}
   \end{algorithm}

   \begin{example}\label{ex:groebnerFan}
     For an easy but clear example, consider the ideal
     \begin{displaymath}
       I:=\langle x+z,y+z \rangle \unlhd \ZZ\llbracket t\rrbracket[x,y,z].
     \end{displaymath}
     Because it is weighted homogeneous w.r.t.~$(-1,0,0,0)\in\RR_{<0}\times\RR^3$ and
     $(0,1,1,1)\in\{0\}\times\RR^3$, its Gr\"obner fan is closed under
     translation by $(-1,0,0,0)$ and invariant under translation by
     $(0,1,1,1)$. We therefore, concentrate on weight vectors on the
     hyperplane $\{0\}\times \RR^2 \times \{0\}$, since any other weight
     vector in the closed lower halfspace can be generated out of them
     via the translations.

     Looking only at potential leading terms of the generators, one might
     be led to believe that the Gr\"obner fan $\Sigma(I)$ restricted to
     $\{0\}\times \RR^2 \times \{0\}$ is of the form as shown in
     Figure~\ref{fig:ex5.8a}
     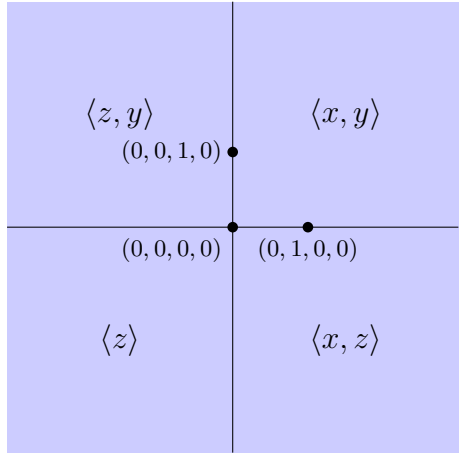
\begin{figure}[h]
       \centering
       \begin{tikzpicture}
         \fill[color=blue!20] (-3,-3) rectangle (3,3);
         \draw
         (0,0) -- (3,0)
         (0,0) -- (0,3)
         (0,0) -- (-3,0)
         (0,0) -- (0,-3);
         \fill (0,0) circle (2pt);
         \node[anchor=north east, font=\scriptsize] at (0,0) {$(0,0,0,0)$};
         \fill (1,0) circle (2pt);
         \node[anchor=north, font=\scriptsize] at (1,0) {$(0,1,0,0)$};
         \fill (0,1) circle (2pt);
         \node[anchor=east, font=\scriptsize] at (0,1) {$(0,0,1,0)$};
         \node at (1.5,1.5) {$\langle x,y\rangle$};
         \node at (-1.5,1.5) {$\langle z,y\rangle$};
         \node at (1.5,-1.5) {$\langle x,z\rangle$};
         \node at (-1.5,-1.5) {$\langle z\rangle$};
       \end{tikzpicture}
       \caption{The Gr\"obner fan $\Sigma(I)$ restricted to $\{0\}\times \RR^2 \times \{0\}$}
       \label{fig:ex5.8a}
     \end{figure}

     Let us use our algorithm to see why this is not the case.
     We start with a random weight vector $u$, say $u=(0,1,1,0)$, and a
     random $t$-local monomial ordering $>$ to be used as
     tiebreaker. Then $\{\underline{x}+z,\underline{y}+z\}$ already is an
     initially reduced standard basis w.r.t.~$>_u$, leading
     terms underlined, so that by Lemma~\ref{lem:Cmembership}
     \begin{displaymath}
       {w'}\in C_u(I) \quad \Longleftrightarrow \quad \begin{cases} \deg_{w'}(x) \geq \deg_{w'}(z)=0, \\ \deg_{w'}(y)\geq \deg_{w'}(z)=0. \end{cases}
     \end{displaymath}
     Hence, $C_u(I)$ is the upper left quadrant of the image above, with two facets available for the traversal.
     \begin{figure}[h]
       \centering
       \begin{tikzpicture}
         \fill[color=blue!20] (0,0) rectangle (3,3);
         \draw[red]
         (0,0) -- (3,0)
         (0,0) -- (0,3);
         \draw[draw opacity=0]
         (0,0) -- (-3,0);
         \fill (0,0) circle (2pt);
         \fill (1,0) circle (2pt);
         \fill (0,1) circle (2pt);
         \node at (1.5,1.5) {$\langle x,y\rangle$};
       \end{tikzpicture}
       \caption{The first cone in the restricted Gr\"obner fan}
       \label{fig:ex5.8b}
     \end{figure}
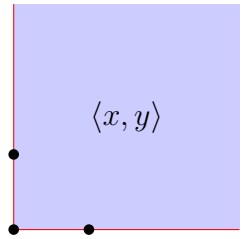
     Picking $\tau$ to be the upper ray of $C_u(I)$, $w=(0,0,1,0)$ a
     relative interior point inside  it and $v=(0,-1,0,0)$ an outer
     normal vector on it, we see that $\initial_w(x+z)=\underline{z}+x$
     and $\initial_w(y+z)=\underline{y}$ already form an initially
     reduced standard basis of $\initial_w(I)$ w.r.t.~$>_{(w,v)}$. Therefore, this standard basis of $\initial_w(I)$ lifts
     again to the very same standard basis
     $\{\underline{z}+x,\underline{y}+z\}$ of $I$ for the adjacent
     ordering.

     However that standard basis is not initially reduced anymore, and a
     quick calculation yields the initially reduced standard basis
     $\{\underline{z}+x, \underline{y}-x\}$ and hence
     \begin{displaymath}
       w'\in C_{>_{(w,v)}}(I) \quad \Longleftrightarrow \quad \begin{cases} 0=\deg_{w'}(z) \geq \deg_{w'}(x), \\ \deg_{w'}(y)\geq \deg_{w'}(x). \end{cases}
     \end{displaymath}
     \begin{figure}[h]
       \centering
       \begin{tikzpicture}
         \fill[color=blue!20] (0,0) rectangle (3,3);
         \fill[color=blue!20] (-3,-3) -- (0,0) -- (0,3) -- (-3,3) -- cycle;
         \draw[red]
         (0,0) -- (3,0)
         (0,0) -- (-3,-3);
         \draw
         (0,0) -- (0,3);
         \fill (0,0) circle (2pt);
         \fill (1,0) circle (2pt);
         \fill (0,1) circle (2pt);
         \node at (1.5,1.5) {$\langle x,y\rangle$};
         \node at (-1.5,0.75) {$\langle z,y\rangle$};
       \end{tikzpicture}
       \caption{The first two cones in the restricted Gr\"obner fan}
       \label{fig:ex5.8c}
     \end{figure}
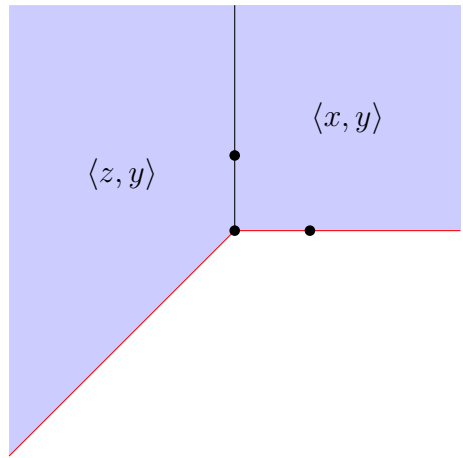

     Let $\tau$ be the lower ray of our new Gr\"obner
     cone (see Figure~\ref{fig:ex5.8c}), $w=(0,-1,-1,0)$ a relative interior point and $v=(0,1,-1,0)$
     an outer normal vector. We see that $\initial(z+x)=\underline{z}$
     and $\initial_w(y-x)=-\underline{x}+y$ already form an initially
     reduced standard basis of $\initial_w(I)$ w.r.t.~$>_{(w,v)}$, which is why it will lift again to the same standard
     basis $\{\underline{z}+x,-\underline{x}+y\}$ of $I$ for the adjacent
     ordering.

     As before, this standard basis is not initially reduced anymore, and
     a quick calculation yields the initially reduced standard basis
     $\{\underline{z}+y,-\underline{x}+y\}$, which means
     \begin{displaymath}
       w'\in C_{>_{(w,v)}}(I) \quad \Longleftrightarrow \quad \begin{cases} 0=\deg_{w'}(z) \geq \deg_{w'}(y), \\ \deg_{w'}(x)\geq \deg_{w'}(y). \end{cases}
     \end{displaymath}
     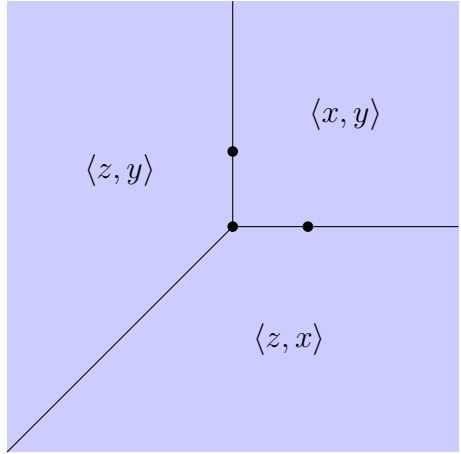
\begin{figure}[h]
       \centering
       \begin{tikzpicture}
         \fill[color=blue!20] (-3,-3) rectangle (3,3);
         \draw
         (0,0) -- (3,0)
         (0,0) -- (-3,-3)
         (0,0) -- (0,3);
         \fill (0,0) circle (2pt);
         \fill (1,0) circle (2pt);
         \fill (0,1) circle (2pt);
         \node at (1.5,1.5) {$\langle x,y\rangle$};
         \node at (-1.5,0.75) {$\langle z,y\rangle$};
         \node at (0.75,-1.5) {$\langle z,x\rangle$};
       \end{tikzpicture}
       \caption{The Gr\"obner fan $\Sigma(I)$ restricted to $\{0\}\times \RR^2 \times \{0\}$}
       \label{fig:ex5.8d}
     \end{figure}
     Figure~\ref{fig:ex5.8d} then shows how the Gr\"obner fan $\Sigma(I)$ actually looks
     like. The misconception at the beginning of the example was due to
     the oversight that $\initial_w(x+z)=\initial_w(y+z)=z$ do not
     generate $\initial_w(I)$, because $\{x+\underline{z},
     y+\underline{z}\}$ is no initially reduced standard basis for
     $>_w$.
   \end{example}

   \begin{remark}
     As we have already remarked, our main interest lies in the
     computation of tropical varieties over the $p$-adic numbers (see
     e.g.~Section~\ref{sec:initiallyreduced}). For
     this we assume that $R=\ZZ$ and $I$ contains the polynomial $p-t$
     for some prime number $p$. At the beginning of this section we
     have mentioned that the traversal algorithm has the advantage
     that a standard basis of the ideal has to computed from scratch
     only once. All intermediate steps comprise of computing standard
     bases of initial ideals (see Algorithm~\ref{alg:flip}), which are
     much simpler since they are 
     weighted homogeneous, and lifting those via computing standard
     representations (see Algorithm~\ref{alg:witness}). A priori, all
     these computations 
     are computations over the integers as base ring, which is
     more expensive than computing over base fields. However, all the
     initial ideals involved contain the 
     prime number $p=\initial_w(p-t)$, and hence most computations can
     actually be done over the finite field $\ZZ/\langle
     p\rangle$. This reduces the overall cost drastically.
   \end{remark}


\providecommand{\bysame}{\leavevmode\hbox to3em{\hrulefill}\thinspace}

\end{document}